\documentclass[12pt]{article}
\usepackage{amsmath}
\usepackage{amssymb}
\usepackage{latexsym}
\usepackage{amsthm}
\usepackage{tabularx}
\usepackage{booktabs}
\usepackage{mathrsfs}
\usepackage{enumitem}
\usepackage{ytableau}
\usepackage{graphicx}
\usepackage{algorithm,algorithmic}
\usepackage{bbm}
\usepackage[colorlinks,
            linkcolor=red,
            anchorcolor=blue,
            citecolor=green]{hyperref}

\usepackage{epic}

\usepackage{amsmath}
\usepackage{tikz}
\usepackage{mathdots}
\usepackage{yhmath}
\usepackage{cancel}
\usepackage{color}
\usepackage{siunitx}
\usepackage{array}
\usepackage{multirow}
\usepackage{amssymb}
\usepackage{gensymb}
\usepackage{tabularx}
\usepackage{extarrows}
\usepackage{booktabs}
\usetikzlibrary{fadings}
\usetikzlibrary{patterns}
\usetikzlibrary{shapes}

\newtheorem{thm}{Theorem}[section]
\newtheorem{lem}[thm]{Lemma}
\newtheorem{prop}[thm]{Proposition}
\newtheorem{cor}[thm]{Corollary}
\newtheorem{conj}[thm]{Conjecture}
\newtheorem{rem}[thm]{Remark}

\newcommand{\VRep}{{\rm VRep}}
\newcommand{\Rep}{{\rm Rep}}

\newcommand{\ch}{{\rm ch}}
\newcommand{\Ind}{{\rm Ind}}
\newcommand{\q}{$q$}

\newcommand{\rk}{{\mathrm{rk}\,}}
\newcommand{\OS}{{\mathrm{OS}}}
\newcommand{\GL}{{\mathrm{GL}}}

\allowdisplaybreaks
\begin{document}
\begin{center}
{\large \bf  Induced log-concavity of equivariant matroid invariants
}
\end{center}

\begin{center}
Alice L.L. Gao$^1$, Ethan Y.H. Li$^2$, Matthew H.Y. Xie$^3$, Arthur L.B. Yang$^{4}$ and Zhong-Xue Zhang$^{5}$\\[6pt]

$^{1}$School of Mathematics and Statistics,\\
Northwestern Polytechnical University, Xi'an, Shaanxi 710072, P.R. China

$^{2}$School of Mathematics and Statistics,\\
Shaanxi Normal University, Xi'an, Shaanxi 710119, P.R. China

$^{3}$College of Science, \\
Tianjin University of Technology, Tianjin 300384, P. R. China

$^{4,5}$Center for Combinatorics, LPMC\\
Nankai University, Tianjin 300071, P. R. China\\[6pt]

Email: $^{1}${\tt llgao@nwpu.edu.cn},
       $^{2}${\tt yinhao\_li@snnu.edu.cn},
      $^{3}${\tt xie@email.tjut.edu.cn},
     $^{4}${\tt yang@nankai.edu.cn},
     $^{5}${\tt zhzhx@mail.nankai.edu.cn}
\end{center}
%mmmmmmmmmmm
\noindent\textbf{Abstract.}
Inspired by the notion of equivariant log-concavity, we introduce the concept of induced log-concavity for a sequence of representations of a finite group.
For an equivariant matroid equipped with a symmetric group action or a finite general linear group action, we transform the problem of proving the induced log-concavity of matroid invariants to that of proving the Schur positivity of symmetric functions. We prove the induced log-concavity of the equivariant Kazhdan-Lusztig polynomials of $q$-niform matroids equipped with the  action of a finite general linear group, as well as that of  the equivariant Kazhdan-Lusztig polynomials of uniform matroids equipped with the action of a symmetric group.
As a consequence of the former, we obtain the log-concavity of Kazhdan-Lusztig polynomials of $q$-niform matroids, thus providing further positive evidence for Elias, Proudfoot and Wakefield's log-concavity conjecture on the matroid Kazhdan-Lusztig polynomials. From the latter we obtain the log-concavity of Kazhdan-Lusztig polynomials of uniform matroids, which was recently proved by Xie and Zhang by using a computer algebra approach.  We also establish the induced log-concavity of the equivariant characteristic polynomials and the equivariant inverse Kazhdan-Lusztig polynomials for $q$-niform matroids and uniform matroids.

\noindent \emph{AMS Classification 2020:} {05B35, 05E05, 20C30} % 05B35 Primary


\noindent \emph{Suggested running title:} Induced log-concavity of equivariant matroid invariants

\noindent \emph{Corresponding Author:} Arthur L.B. Yang, yang@nankai.edu.cn

%\noindent \emph{Conflict of interest:}
%The authors declare that they have no conflict of interest to this work.

\section{Introduction}

The main objective of this paper is to provide a general framework to study the
log-concavity of matroid invariants by introducing a new concept which generalizes equivariant log-concavity, called induced log-concavity. We would like to point out that this generalization is nontrivial and provides more freedom to use deep theory and tools from other branches of mathematics. As will be shown below, this new concept allows us to give a first proof of the log-concavity conjecture of the Kazhdan-Lusztig polynomials of $q$-niform matroids. The reason that we can not use the equivariant log-concavity to do so for the moment is that the proof will involve the Kronecker product of Schur functions, which is substantially difficult to understand and whose combinatorial interpretation remains as one of the central open problems in algebraic combinatorics. While the induced log-concavity enables us to reduce the log-concavity conjecture for $q$-niform matroids to certain Schur positivity problems concerning the well-understood ordinary product of Schur functions. It is worth mentioning that our approach simultaneously establishes the log-concavity of the Kazhdan-Lusztig polynomials of uniform matroids. We hope that this conceptual proposal could be applicable to more occasions.

Now let us first review some relevant background. Recall that a finite sequence $(a_0,\, a_1,\, \ldots,\, a_n)$ of real numbers is said to be unimodal if $a_0\le a_1\le\cdots\le a_i\ge a_{i+1}\ge\cdots\ge a_n$ for some $0\le i\le n$, and it is said to be log-concave if $a_i^2\ge a_{i-1}a_{i+1}$ for any $1\le i\le n-1$.
%\begin{itemize}
%    \item log-concave if $a_i^2\ge a_{i-1}a_{i+1}$ for any $1\le i\le n-1$;
%    \item ultra log-concave if $\frac{a_i^2}{\binom{n}{i}^2}\ge \frac{a_{i-1}}{\binom{n}{i-1}}\frac{a_{i+1}}{\binom{n}{i+1}}$ for any $1\le i\le n-1$;
%    \item unimodal if $a_0\le a_1\le\cdots\le a_i\ge a_{i+1}\ge\cdots\ge a_n$ for some $0\le i\le n$.
%\end{itemize}
We say that $(a_0,\,a_1,\,\ldots,\,a_n)$ has no internal zeros if there do not exist integers $0\le i<j<k\le n$ satisfying $a_i\neq 0,\, a_j=0,\, a_k\neq 0$.
We also say that a polynomial
$a_0+a_1t+\cdots +a_n t^n$ with real coefficients has a certain property if its coefficient sequence
$(a_0,\,a_1,\,\ldots,\,a_n)$ does. It is clear that a nonnegative log-concave sequence with no internal zero must be unimodal.
Unimodal and log-concave sequences and polynomials are
ubiquitous in combinatorics, geometry, probability, and statistics.
For more information on the subject of unimodality and log-concavity, we refer the reader to
Stanley \cite{Stanley-1989}, Brenti \cite{Brenti-1994} and Br\"and\'en \cite{Branden-2014}.
Although various methods and theories have been developed for proving unimodality and log-concavity, there are still many challenging conjectures.
% \textcolor{red}{full of challenges} .

In recent years the log-concavity of matroid invariants has received considerable research attention, and for significant progress on some outstanding problems see
Huh \cite{Huh-2012}, Huh and Katz \cite{HK-2012},
Adiprasito, Huh, and Katz \cite{AHK-2018}, Br\"and\'en and Huh \cite{BH-2020},  Adila, Denham, and Huh \cite{ADH-2020}, and Braden, Huh, Matherne, Proudfoot and Wang \cite{BHMPW-2020}. This paper is mainly motivated by the log-concavity conjecture of matroid Kazhdan-Lusztig polynomials due to Elias, Proudfoot and Wakefield \cite{EPW-2016}.

Both the matroid Kazhdan-Lusztig polynomials and the classical Kazhdan-Lusztig polynomials for Coxeter groups \cite{KL-1979} are special cases of Kazhdan-Lusztig-Stanley polynomials.
Polo \cite{Polo-1999} showed that any polynomial with nonnegative coefficients and constant term $1$ appears as a Kazhdan-Lusztig polynomial associated to some pair of elements in some symmetric group. Hence, the classical Kazhdan-Lusztig polynomials need not %\textcolor{red}{to}
be log-concave, while Elias, Proudfoot and Wakefield \cite{EPW-2016} conjectured that all matroid Kazhdan-Lusztig polynomials are log-concave. Gedeon, Proudfoot, and Young \cite{GPY-2017real} conjectured that each matroid Kazhdan-Lusztig polynomial has only real zeros, and showed their conjecture is valid for the graphical matroids associated with cycle graphs.
By Newton's inequality, the latter conjecture implies the former log-concavity conjecture.

Elias, Proudfoot and Wakefield's log-concavity conjecture remains open, and it was confirmed for several families of matroids, such as uniform matroids by Xie and Zhang \cite{XZ-2021}. Gedeon, Proudfoot, and Young's  real-rootedness conjecture was confirmed for some uniform matroids by Gao, Lu, Xie, Yang and Zhang \cite{GLXYZ-2018}, and for fan matroids, wheel matroids, and whirl matroids by Lu, Xie and Yang \cite{LXY-2018}.
The primary goal of this paper is to prove
Elias, Proudfoot and Wakefield's log-concavity conjecture for $q$-niform matroids, which were studied by Proudfoot in \cite{P-2019}.

In order to study the matroid Kazhdan-Lusztig polynomials,
other two families of matroid invariants were introduced, including the inverse Kazhdan-Lusztig polynomials defined by Gao and Xie \cite{GX-2021},  and the $Z$-polynomials defined by Proudfoot, Xu and Young \cite{PXY-2018}.
%\textbf{(originally suggesting the study of the $Z$-polynomial by Sara Billey) }
Gao and Xie \cite{GX-2021} conjectured that all inverse Kazhdan-Lusztig polynomials of matroids are log-concave and they proved this conjecture for uniform matroids. %Recently, \textcolor{blue}{Xie and Zhang}  \cite{XZ-2022} proved the log-concavity of inverse Kazhdan-Lusztig polynomials for sparse paving matroids.
Proudfoot, Xu and Young \cite{PXY-2018} conjectured that all $Z$-polynomials are log-concave and even have only real zeros and they verified the real-rootedness for modular matroids. The real-rootedness of $Z$-polynomials were further confirmed for fan matroids, wheel matroids, and whirl matroids by Lu, Xie and Yang \cite{LXY-2018}, and for some uniform matroids by Gao, Lu, Xie, Yang and Zhang \cite{GLXYZ-2018}.

By imposing a group action on a matroid, Gedeon, Proudfoot and Young \cite{GPY-2017} defined the equivariant matroid Kazhdan-Lusztig polynomials, which form another important family of matroid invariants.
Both inverse Kazhdan-Lusztig polynomials and $Z$-polynomials have their equivariant counterparts; see
Proudfoot \cite{P-2021} and Proudfoot, Xu and Young \cite{PXY-2018}.
The equivariant Kazhdan-Lusztig polynomials provide more powerful tools to study the
ordinary Kazhdan-Lusztig polynomials.
Among these, the concept of equivariant log-concavity,
which was proposed by Gedeon, Proudfoot, and Young \cite{GPY-2017}, serves as a natural extension of the notion of log-concavity for matroid invariants.

To recall the definition of equivariant log-concavity, we shall adopt most of the notations and symbols in \cite{GPY-2017}. For a finite group $W$,
let $\Rep(W)$ denote the set of honest representations of $W$ over $\mathbb{C}$, and
let $\VRep(W)$ denote the ring of virtual representations, namely, the formal difference (with respect to direct sum) of two honest representations.
A sequence $(C_i)_{i\ge 0}$ in $\VRep(W)$ is said to be equivariantly log-concave if $C_{i}\otimes C_{i}-C_{i-1}\otimes C_{i+1}\in\Rep(W)$ for all $i > 0$, and it is said to be strongly equivariantly log-concave if $C_{i}\otimes C_{j}-C_{i-1}\otimes C_{j+1}\in\Rep(W)$ for all $1\le i\le j$, where the symbol $\otimes$ denotes the (internal) tensor product of representations. Sometimes we use $C_{i}\otimes C_{j}\supseteq C_{i-1}\otimes C_{j+1}$ to represent $C_{i}\otimes C_{j}-C_{i-1}\otimes C_{j+1}\in\Rep(W)$.
%In such cases, we also say that the graded representation $\oplus_{i\geq 0} C_i$ is equivariantly log-concave or strongly equivariantly log-concave.
Letting $\dim\,C_i$ denote the dimension of $C_i$, it is clear that the equivariant log-concavity of
$(C_i)_{i\ge 0}$ implies the log-concavity of
$(\dim\,C_i)_{i\ge 0}$. Note that the notion of equivariant log-concavity can be carried over verbatim to
polynomials with coefficients in $\VRep(W)$, or to graded representations of $W$.

Suppose that $M$ is a matroid equipped with an action of
$W$, denoted by $W\curvearrowright M$.
Gedeon, Proudfoot, and Young \cite{GPY-2017} conjectured that for any equivariant matroid $W\curvearrowright M$ both the equivariant characteristic polynomial and  the equivariant Kazhdan-Lusztig polynomial are equivariantly log-concave, and they also made partial progress for equivariant uniform matroids.
%. They confirmed the equivariant log-concavity of equivariant characteristic polynomials for any equivariant matroid $\mathfrak{S}_{m+d}\curvearrowright {U_{m,d}}$, where $\mathfrak{S}_{m+d}$ is a symmetric group on $\{1,2,\ldots,m+d\}$ and $U_{m,d}$ is a uniform matroid of rank $d$ on ground set $\{1,2,\ldots,m+d\}$.
%They also verified the equivariant log-concavity of equivariant Kazhdan-Lusztig polynomials for $\mathfrak{S}_{m+d}\curvearrowright {U_{m,d}}$ with $m,d\le 15$ by a computer.
Proudfoot, Xu and Young \cite{PXY-2018} also conjectured all equivariant $Z$-polynomials are strongly equivariantly log-concave, and verified their conjecture for certain $q$-analogue of equivariant Boolean matroids.
%Some algebras can be seen as graded representations of finite groups and such algebra is called (strongly) equivariantly log-concave if the sequence of its graded representations in decreasing order is (strongly) equivariantly log-concave.
Matherne, Miyata, Proudfoot and Ramos \cite{MMPR-2021} provided more equivariantly log-concave conjectural examples associated with Orlik-Solomon algebras of matroids, Cordovil algebras of oriented matroids, and  Orlik-Terao algebras of hyperplane arrangements. They also gave computer assisted proofs to their conjectures for some special cases in low degrees by using the theory of representation stability.

Motivated by the notion of equivariant log-concavity and its connections to the ordinary log-concavity, we introduce a parallel concept, called induced log-concavity, whose definition is more subtle. Given a sequence $(C_i)_{i\ge 0}$ of virtual representations of a finite group $W$, the notion of induced log-concavity will be concerned with the external tensor product of $C_i$ and $C_j$, denoted by $C_i\boxtimes C_j$, which is considered as a representation of $W\times W$ and  different from the internal tensor product used in the definition of equivariant log-concavity. If there exist some finite group $G$, a subgroup $H$ of $G$, and a group homomorphism
$\phi\,:\, H\rightarrow W\times W$ such that
\begin{align*}
    \Ind_{H}^{G}\, (C_i\boxtimes C_i) - \Ind_{H}^{G}\, (C_{i-1}\boxtimes C_{i+1})\in\Rep(G)
\end{align*}
for any $i\ge 1$, where the external tensor product $C_i\boxtimes C_j$ is naturally considered as a representation of $H$ via the pullback of $\phi$,
then $(C_i)_{i\ge 0}$ is called inductively log-concave with respect to $G$, $H$ and $\phi$. If $G$, $H$ and $\phi$ is obvious from the context, we simply say that $(C_i)_{i\ge 0}$ is inductively log-concave. Furthermore, if
\begin{align*}
    \Ind_{H}^{G}\, (C_i\boxtimes C_j) - \Ind_{H}^{G}\, (C_{i-1}\boxtimes C_{j+1})\in\Rep(G)
\end{align*}
for any $1\le i\le j$, then $(C_i)_{i\ge 0}$ is called strongly inductively log-concave. Later we also use $    \Ind_{H}^{G}\, (C_i\boxtimes C_j)\supseteq \Ind_{H}^{G}\, (C_{i-1}\boxtimes C_{j+1})
$ to represent
$    \Ind_{H}^{G}\, (C_i\boxtimes C_j) - \Ind_{H}^{G}\, (C_{i-1}\boxtimes C_{j+1})\in\Rep(G)
$.
Note that if we take $G=H=W$ and take $\phi=\mathrm{diag}: W\rightarrow W\times W$ to be the diagonal embedding,
then the induced log-concavity reduces to the equivariant log-concavity. Thus the notion of induced log-concavity provides a more general framework to study the ordinary log-concavity of matroid invariants. The notion of induced log-concavity can also be carried over verbatim to polynomials with coefficients in $\VRep(W)$, or to graded representations of $W$.

In this paper we shall explore the use of induced log-concavity for proving the log-concavity of Kazhdan-Lusztig polynomials of uniform matroids and $q$-niform matroids. Our contributions are the following.

\begin{itemize}
\item[(I)] Two simple ways to generate new inductively log-concave polynomials from old ones. In particular, these can be used to produce equivariantly log-concave polynomials. (See Section \ref{sect-2}.)

\item[(II)] Schur positivity of some differences of products of Schur functions. (See Section \ref{sect-3}.)

\item[(III)] Induced log-concavity of equivariant characteristic polynomials,
Kazhdan-Lusztig polynomials, and inverse Kazhdan-Lusztig polynomials of $q$-niform matroids and uniform matroids. As a corollary, we obtain the log-concavity of the Kazhdan-Lusztig polynomials of $q$-niform matroids. (See Section \ref{sect-4}.)
\end{itemize}
%
%\begin{thm}\label{thm-main-uq-log}
%    The Kazhdan-Lusztig polynomial of $q$-niform matroids are log-concave.
%\end{thm}

%This paper is organized as follows.
%Section \ref{sect-2} is devoted to the study of some properties of induced log-concavity, particularly involving the representations of symmetric groups and the unipotent representations of finite general linear groups.
%In Section \ref{sect-3} we give some Schur positivity results, which will be used in Section \ref{sect-4} to prove the induced log-concavity related to symmetric groups and finite general linear groups.
%In Section \ref{sect-4} we prove the induced log-concavity of equivariant characteristic polynomials,
%equivariant Kazhdan-Lusztig polynomials, and equivariant inverse Kazhdan-Lusztig polynomials of $q$-niform matroids and uniform matroids, and we also study the induced log-concavity of the associated equivariant $Z$-polynomials.

\section{Induced log-concavity}\label{sect-2}

The aim of this section is two-fold. First, we present some properties of induced log-concavity of virtual representations, parallel to those of ordinary log-concavity of real numbers. Secondly, we show some connection between the induced log-concavity of representations of symmetric groups and that of unipotent representations of finite general linear groups.

Since we are mainly concerned with the log-concavity of polynomials, all log-concavity results of this section
are stated in terms of polynomials instead of sequences.
Given a finite group $W$, let $\mathrm{VRep}(W)[t]$ be the polynomial ring in the variable $t$ over
$\mathrm{VRep}(W)$.
%Recall that the addition and multiplication in $\mathrm{VRep}(W)$ are formally defined by the direct sum ``$\oplus$'' and internal tensor product ``$\otimes$''.
%We use the usual symbol ``+'' and ``$\cdot$'' to denote
%the addition and multiplication in $\mathrm{VRep}(W)[t]$,
%to distinguish these operations from those used in $\mathrm{VRep}(W)$. Precisely,
%For any two polynomials
%$$J(t)=\sum_{i=0}^n C_it^i,\quad \, g(t)=\sum_{i=0}^m D_it^i$$
%in $\mathrm{VRep}(W)[t]$, we let
%\begin{align*}
%    J(t)+g(t)&=\sum_{i=0}^{\max\{m,n\}} (C_i\oplus D_i) t^i,\\
%    J(t)\cdot g(t)&=\sum_{i=0}^{m+n} \left(\bigoplus_{j=0}^i C_j\otimes D_{i-j}\right) t^i,
%\end{align*}
%where we set $C_i=0$ for $i>n$ and $D_j=0$ for $j>m$.
For notational convenience, we use $t+\tau$ to denote the polynomial $\tau t+\tau$ in $\mathrm{VRep}(W)[t]$, where $\tau$ denotes the trivial representation of $W$.

Suppose that $J(t)\in \mathrm{VRep}(W)[t]$ is inductively log-concave with respect to some groups $G,H$ and some group homomorphism $\phi\,:\, H\rightarrow W\times W$.
In the following we shall use $J(t)$ to generate new inductively log-concave polynomials with respect to the same triple $(G,H,\phi)$. To this end, we require that the triple $(G,H,\phi)$ should satisfy the following symmetric property:

($\diamondsuit$) For any two virtual representations $C,D\in \mathrm{VRep}(W)$, the induced representation $\Ind_H^G\, (C\boxtimes D)$ is isomorphic to $\Ind_H^G\, (D\boxtimes C)$.

For example, the triple $(W,W,\mathrm{diag}: W\rightarrow W\times W)$ satisfies the above property, and thus the following two propositions hold in particular for equivariant log-concavity. It is well known that if $J(t)\in\mathbb{R}[t]$ is a log-concave polynomial with nonnegative coefficients, then $(t+1)J(t)$ is also log-concave. Motivated by this, we obtained the following result.

\begin{prop}\label{prop-times (t+1)}
 Suppose that $W$ is a finite group, the triple $(G,H,\phi)$ has property ($\diamondsuit$), and $J(t)$ is a polynomial in $\mathrm{VRep}(W)[t]$ with coefficients being honest representations. If $J(t)$ is strongly inductively log-concave with respect to a triple $(G,H,\phi)$, then so is the polynomial
  $(t+\tau)\cdot J(t)$.
\end{prop}

\begin{proof} Assume that
$J(t)=\sum_{i=0}^n C_it^i$. By the hypothesis, for any $1\le i\le j\le n-1$ we have
    \begin{align}\label{equ-ind-strongly-tool}
        \Ind_H^G\, (C_i\boxtimes C_j)-\Ind_H^G\, (C_{i-1}\boxtimes C_{j+1})\in\Rep(G).
    \end{align}
Since $C_i\otimes \tau = C_i$ in the sense of isomorphism,
we get
    \begin{align*}
        (t+\tau)\cdot J(t)
        =\sum_{k=0}^{n+1} (C_{k-1}+ C_{k}) t^{k},
    \end{align*}
where we set $C_{-1}=C_{n+1}=0$. Now, for any  $1\le i\le j\le n$, one can verify that
%observe that
%\begin{align*}
%(C_{i-1}+ C_{i})\boxtimes (C_{j-1}+ C_{j})
%=&(C_{i-1}\boxtimes C_{j-1})+(C_{i-1}\boxtimes C_{j})+(C_{i}\boxtimes C_{j-1})+(C_{i}\boxtimes C_{j}),\\
%(C_{i-2}+ C_{i-1})\boxtimes (C_{j}+ C_{j+1})
%=&(C_{i-2}\boxtimes C_{j})+(C_{i-2}\boxtimes C_{j+1})+(C_{i-1}\boxtimes C_{j})+(C_{i-1}\boxtimes C_{j+1}),
%\end{align*}
%and hence
\begin{align*}
   &\Ind_H^G\, \left((C_{i-1}+ C_{i})\boxtimes (C_{j-1}+ C_{j})\right)-
      \Ind_H^G\, \left((C_{i-2}+ C_{i-1})\boxtimes (C_{j}+ C_{j+1})\right)\\
   &=\left(\Ind_H^G\,(C_{i-1}\boxtimes C_{j-1})-\Ind_H^G\,(C_{i-2}\boxtimes C_{j})\right)+
       \left(\Ind_H^G\,(C_{i}\boxtimes C_{j})-\Ind_H^G\,(C_{i-1}\boxtimes C_{j+1})\right)\\
   &+ \left(\Ind_H^G\,(C_{i}\boxtimes C_{j-1})-\Ind_H^G\,(C_{i-1}\boxtimes C_{j})\right)+
       \left(\Ind_H^G\,(C_{i-1}\boxtimes C_{j})-\Ind_H^G\,(C_{i-2}\boxtimes C_{j+1})\right),
\end{align*}
    which is an honest representation of $G$ by \eqref{equ-ind-strongly-tool}.
    A little caution is needed when dealing with the case of $i=j$, for which the third direct summand vanishes due to property ($\diamondsuit$) and the remaining direct summands are honest by \eqref{equ-ind-strongly-tool}. The proof is complete.
\end{proof}

Another known observation on log-concavity is that if $J(t)\in\mathbb{R}[t]$ is a log-concave polynomial with nonnegative coefficients, then so is $J(t+1)$; see \cite{Brenti-1994,Lenz-2013}. Inspired by this fact, we obtain the following result.

\begin{prop}\label{prop-J(t+1)}
   Suppose that $W$ is a finite group, the triple $(G,H,\phi)$ has property ($\diamondsuit$), and $J(t)$ is a polynomial in $\mathrm{VRep}(W)[t]$ with coefficients being honest representations. If $J(t)$ is strongly inductively log-concave with respect to a triple $(G,H,\phi)$, then so is the polynomial $J(t+\tau)$.
\end{prop}

\begin{proof}
    The proof is by induction on the degree $n$ of $J(t)$. We may assume that $n\geq 2$. Suppose that
      $$J(t)=C_0+C_1t+C_2t^2$$
    is strongly inductively log-concave with respect to  $(G,H,\phi)$, namely
    \begin{align}\label{equ-base-case}
        \Ind_H^G\, (C_1\boxtimes C_1)-\Ind_H^G\, (C_{0}\boxtimes C_{2})\in\Rep(G).
    \end{align}
    Note that
       $$J(t+\tau)=(C_0+ C_1+ C_2)+(C_1+ 2 C_2)t+C_2t^2.$$
    We need to prove that
    \begin{align*}%\label{equ-ind-strongly-tool}
        \Ind_H^G\, ((C_1+ 2C_2)\boxtimes (C_1+ 2C_2))-
           \Ind_H^G\, ((C_0+ C_1+ C_2)\boxtimes C_{2})\in\Rep(G).
    \end{align*}
    By using \eqref{equ-base-case}, one can verify that
    \begin{align*}
        \Ind_H^G\, ((C_1+ 2C_2)\boxtimes (C_1+ 2C_2))&\supseteq \Ind_H^G\, (C_1\boxtimes C_1)+ \Ind_H^G\, (C_1\boxtimes C_2)+ \Ind_H^G\, (C_2\boxtimes C_2)\\
        &\supseteq \Ind_H^G\, (C_0\boxtimes C_2)+ \Ind_H^G\, (C_1\boxtimes C_2)+
          \Ind_H^G\, (C_2\boxtimes C_2)\\
        &=\Ind_H^G\, ((C_0+ C_1+ C_2)\boxtimes C_2).
    \end{align*}
    This completes the proof of the base case.

    Now assume the assertion is true for polynomials of degree less than $n+1$. We proceed to prove the assertion for a strongly inductively log-concave polynomial $J(t)$ of degree $n+1$ in $\mathrm{VRep}(W)[t]$. Suppose that $J(t)=\sum_{i=0}^{n+1}C_it^{i}$ with $C_i\in \mathrm{Rep}(W)$. Then, for any $1\leq i\leq j\leq n$, we have
    \begin{align}\label{eq-induct-step}
       \Ind_H^G\, (C_i\boxtimes C_j)-\Ind_H^G\, (C_{i-1}\boxtimes C_{j+1})\in\Rep(G).
    \end{align}
    If we write
    $$K(t)=\sum_{i=0}^{n}C_{i+1}(t+\tau)^{i}=\sum_{i=0}^{n}D_i t^{i},$$
    then
    \begin{align*}
        J(t+\tau)=C_{0}+(t+\tau)\cdot K(t), \quad D_i=\sum_{j=i}^n \binom{j}{i}C_{j+1} \mbox{ for } 0\leq i\leq n.
    \end{align*}
%By the induction hypothesis and Proposition \ref{prop-times (t+1)}, we know that the polynomial
%$(t+\tau)\cdot g(t)$ is strongly inductively log-concave.
%If we let $D_{-1}=D_{n+1}=0$ and write $(t+\tau)\cdot g(t)=\sum_{i=0}^{n+1}E_i t^i$, then $E_i=D_i\oplus D_{i-1}$. Since $J(t+\tau)$ and $(t+\tau)\cdot g(t)$ have the same coefficients except for the constant term, it remains to show that for any $1\leq i\leq n$
%\begin{align*}
%  \Ind_H^G\, (E_1\boxtimes E_i)-\Ind_H^G\,((C_0\oplus E_0)\boxtimes E_{i+1})\in\Rep(G).
%\end{align*}
    Observe that if
     $$C_{0} + t\cdot J(t+\tau)=C_{0} + t\cdot ( C_{0}+ (t+\tau)\cdot K(t))
      =(t+\tau)\cdot(C_{0}+ t\cdot K(t))$$
    is strongly inductively log-concave, then so is $J(t+\tau)$.
    Thus it suffices to show that $C_{0} + t\cdot J(t+\tau)$
    is strongly inductively log-concave.
    By Proposition \ref{prop-times (t+1)}, we only need to prove that $C_{0}+ t\cdot K(t)$ is strongly inductively log-concave. By the induction hypothesis, we know that $K(t)$, and hence $t\cdot K(t)$, is strongly inductively log-concave with respect to  $(G,H,\phi)$.
    Thus, we only need to check that for any $0\le i\le n-1$
    \begin{align*}
        \Ind_H^G\, (D_0\boxtimes D_i)-\Ind_H^G\,(C_0\boxtimes D_{i+1})\in\Rep(G).
    \end{align*}
We find that
    \begin{align*}
        \Ind_H^G\, (D_0\boxtimes D_i)&=\Ind_H^G\, \left(\left(\sum_{j_1=0}^n C_{j_1+1}\right)\boxtimes \left(\sum_{j_2=i}^n \binom{j_2}{i} C_{j_2+1}\right)\right)\\[5pt]
        &=\sum_{j_1=0}^n\sum_{j_2=i}^n \Ind_H^G\, \left( C_{j_1+1}\boxtimes \binom{j_2}{i}C_{j_2+1}\right)\\[5pt]
        &=\sum_{\ell=1}^{n+1}\sum_{k=\ell+i+1}^{\ell+n+1} \Ind_H^G\, \left( C_{\ell}\boxtimes \binom{k-\ell-1}{i}C_{k-\ell}\right)\\[5pt]
        &\supseteq\sum_{\ell=1}^{n-i}\sum_{k=\ell+i+1}^{n+1} \Ind_H^G\, \left( C_{\ell}\boxtimes \binom{k-\ell-1}{i}C_{k-\ell}\right)\\[5pt]
        &=\sum_{k=i+2}^{n+1}\sum_{\ell=1}^{k-i-1} \Ind_H^G\,\left(C_{\ell}\boxtimes \binom{k-\ell-1}{i}C_{k-\ell}\right)\\[5pt]% \oplus
        %\sum_{k=n+2}^{2n-i+2}\sum_{\ell=k-n-1}^{n-i+1} \Ind_H^G\,(C_{\ell}\boxtimes C_{k-\ell}^{\oplus \tbinom{k-\ell-1}{i}})\\
        &\supseteq \sum_{k=i+2}^{n+1} \Ind_H^G\,\left(C_{0}\boxtimes \left(\sum_{\ell=1}^{k-i-1} \binom{k-\ell-1}{i}\right)C_{k}\right),
\end{align*}
where the last inclusion is derived from   \eqref{eq-induct-step} and $(\diamondsuit)$.

By successively applying
$\binom{k}{i+1}+\binom{k}{i}=\binom{k+1}{i+1}$, we see that $\sum_{\ell=1}^{k-i-1} \tbinom{k-\ell-1}{i}=\tbinom{k-1}{i+1}$. Thus
\begin{align*}
        \Ind_H^G\, (D_0\boxtimes D_i)&\supseteq \sum_{k=i+2}^{n+1} \Ind_H^G\,\left(C_{0}\boxtimes \tbinom{k-1}{i+1} C_{k}\right)\\[5pt]
&=\Ind_H^G\,\left(C_0\boxtimes \left(\sum_{j=i+1}^n \binom{j}{i+1}C_{j+1}\right)\right)=\Ind_H^G\,(C_0\boxtimes D_{i+1}),
    \end{align*}
    as desired.
This completes the proof.
\end{proof}

Now we turn to the second purpose of this section.
As shown before, we are unaware which triple
$(G,H,\phi)$ having property ($\diamondsuit$) is used in the preceding two propositions.
However, the choice of the triple is critical
for practical applications of induced log-concavity.
For our purpose here, we will introduce a feasible way to choose the triple $(G,H,\phi)$ for symmetric groups and finite general linear groups.

Let us first consider the induced log-concavity of representations of a symmetric group. To set up the necessary machinery, we briefly review some basic facts about the representation theory of symmetric groups.
Let $\mathfrak{S}_n$ denote the symmetric group of all permutations of the finite set $\{1,2,\ldots ,n\}$.
The irreducible representations of $\mathfrak{S}_n$ can be indexed by partitions of $n$.
Recall that a partition $\lambda$ of $n$, denoted by $\lambda\vdash n$, is a sequence $(\lambda_1,\lambda_2,\ldots,\lambda_l)$ of weakly decreasing nonnegative integers $\lambda_i$ such that $\sum_{i=1}^l\lambda_i=n$.
We usually use the notation $V_{\lambda}$ to denote the irreducible representation of $\mathfrak{S}_n$ associated with the partition $\lambda$.
Given any integer $k\le n$ and two representations $V\in\Rep(\mathfrak{S}_k),\,V'\in\Rep(\mathfrak{S}_{n-k})$,  define
\begin{align}\label{formula-induced-produce}
    V\ast V':=\Ind_{\mathfrak{S}_k\times \mathfrak{S}_{n-k}}^{\mathfrak{S}_{n}}\left(V\boxtimes V'\right).
\end{align}
There is a natural isomorphism between $\oplus_{n\geq 0}\mathrm{VRep}(\mathfrak{S}_n)$ and the ring $\Lambda_\mathbb{Z}$ of symmetric functions over $\mathbb{Z}$, called the
Frobenius characteristic map $\mathrm{ch}$, which sends
each irreducible representation $V_{\lambda}$ to the Schur function $s_{\lambda}$ and satisfies $\mathrm{ch}(V\ast V')=\mathrm{ch} V\cdot \mathrm{ch} V'$. More information on  $\Lambda_\mathbb{Z}$ and $s_{\lambda}$ will be given in Section \ref{sect-3}.
To study the induced log-concavity of a sequence $(C_i)_{i\ge 0}$ of virtual representations of $\mathfrak{S}_{n}$, we will take $G=\mathfrak{S}_{2n}$, $H=\mathfrak{S}_{n}\times \mathfrak{S}_{n}$, and $\phi=\mathrm{id}$, the identity map. Thus, the sequence
$(C_i)_{i\ge 0}$ is inductively log-concave with respect to $(\mathfrak{S}_{2n}, \mathfrak{S}_{n}\times \mathfrak{S}_{n}, \mathrm{id})$ if
$C_i\ast C_i-C_{i+1}\ast C_{i-1}\in \Rep(\mathfrak{S}_{2n})$ for any $i\geq 1$.
The reason for such a choice is that
the Frobenius characteristic map enables us
to transform the problem of proving the induced log-concavity of certain representations of symmetric groups to that of proving the Schur positivity of the corresponding symmetric functions.

%Later, we will prove the induced log-concavity of some matroids' equivariant invariants of uniform matroids by some Schur-positivity results.

For finite general linear groups, this paper is mainly concerned with the induced log-concavity of unipotent representations.
Given a prime power $q$, let $\mathbb{F}_q$ be a finite field of order $q$, and let $\GL_n(\mathbb{F}_q)$ be the finite general linear group over $\mathbb{F}_q$.
Let $B_n(\mathbb{F}_q)\subseteq \GL_n(\mathbb{F}_q)$ denote the Borel subgroup composed of invertible upper triangular matrices.
The irreducible unipotent representations of $\GL_n(\mathbb{F}_q)$ are the composition factors of the representation $\mathbb{C}[\GL_n(\mathbb{F}_q)/B_n(\mathbb{F}_q)]$.
If a representation is isomorphic to a direct sum of irreducible unipotent representations, then it is called a unipotent representation.
The irreducible unipotent representations of $\GL_n(\mathbb{F}_q)$ can also be indexed by partitions of $n$; see Curtis \cite[Theorem B]{C-1975}. We save the notation $V_{\lambda}(q)$ for the associated irreducible unipotent representation of $\GL_n(\mathbb{F}_q)$ associated with $\lambda$.
For a pair of natural numbers $k\leq n$, let $P_{k,n}(\mathbb{F}_q) \subseteq \GL_{n}(\mathbb{F}_q)$ denote the parabolic subgroup associated with the $\mathrm{Levi}\,\GL_k(\mathbb{F}_q)\times \GL_{n-k}(\mathbb{F}_q)$. Suppose that $V(q)$ is a unipotent representation of $\GL_k(\mathbb{F}_q)$, and $V'(q)$ is a unipotent representation of $\GL_{n-k}(\mathbb{F}_q)$. It is possible to consider $V(q)\boxtimes V'(q)$ as a representation of $P_{k,n}(\mathbb{F}_q)$ by using the natural surjection
$\iota: P_{k,n}(\mathbb{F}_q) \longrightarrow \GL_{k}(\mathbb{F}_q) \times \GL_{n-k}(\mathbb{F}_q)$.
%\begin{align*}
%   \eta_{k,n}: P_{k,n}(\mathbb{F}_q) &\longrightarrow \GL_{k}(\mathbb{F}_q) \times \GL_{n-k}(\mathbb{F}_q)\\
%    % \begin{pmatrix}
%    %     (A_1)_{k\times k} & (A_3)_{k\times n-k}\\
%    %      0  & (A_2)_{n-k\times n-k}
%    % \end{pmatrix}
%    \begin{pmatrix}
%        A_1 & A_3 \\
%         0  & A_2
%    \end{pmatrix}
%    &\longmapsto
%    \begin{pmatrix}
%        A_1 & 0 \\
%         0  & A_2
%    \end{pmatrix},
%\end{align*}
%where $A_1\in\GL_{k}(\mathbb{F}_q)$, $A_2\in\GL_{n-k}(\mathbb{F}_q)$, and $A_3$ is a $k\times (n-k)$ matrix.
The Harish-Chandra induction
is defined by
\begin{align*}%\label{eq-hs-induction}
    V(q)\ast V'(q):=\Ind_{P_{k,n}(\mathbb{F}_q)}^{\GL_{n}(\mathbb{F}_q)}\left(V(q)\boxtimes V'(q)\right).
\end{align*}
It is known that $V(q)\ast V'(q)$ is a unipotent representation of $\GL_{n}(\mathbb{F}_q)$.
%
%Thus there is a one-to-one correspondence between unipotent representations of $\GL_n(\mathbb{F}_q)$ and ordinary representations of $\mathfrak{S}_n$. We also would like to point out that, for any partition $\lambda$ of $n$, Andrews \cite{A-2018} gave a construction of $V(q)_{\lambda}$ similar to the construction of irreducible representation $V_{\lambda}$ via tableaux. Moreover, as shown by the Comparison Theorem below, it makes no difference to consider the irreducible multiplicity of induced representations, whether working with the Harish-Chandra induction for $\GL_n(\mathbb{F}_q)$ or the ordinary induction for
%$\mathfrak{S}_n$.
Let $\mathrm{VURep}(\GL_{n}(\mathbb{F}_q))$ denote the set of virtual unipotent representations of $\GL_{n}(\mathbb{F}_q)$. There is a canonical bijection $\psi$ between $\mathrm{VURep}(\GL_{n}(\mathbb{F}_q))$
and $\mathrm{VRep}(\mathfrak{S}_n)$ by taking  $\psi(V_\lambda(q))=V_{\lambda}$ and then extending linearly. The representation theory of symmetric groups and the unipotent representation theory of finite general linear groups behave very similarly in the sense of the following Comparison Theorem.

\begin{thm}[{\cite[Theorem B]{C-1975}}]\label{thm-Comparison}
Fix a prime power $q$ and a pair of natural numbers $k\le n$, and let $\lambda$, $\mu$ and  $\nu$ be partitions of $n$, $k$ and $n-k$ respectively.
Then the multiplicity of $V_\lambda(q)$ in  $V_{\mu}(q)* V_{\nu}(q)$ is equal to the multiplicity of $V_\lambda$ in  $V_{\mu} * V_{\nu}$.
\end{thm}
To study the induced log-concavity of a sequence $(C_i(q))_{i\ge 0}$ of virtual unipotent representations of $\GL_n(\mathbb{F}_q)$, we will take $G=\GL_{2n}(\mathbb{F}_q)$, $H=P_{n,2n}(\mathbb{F}_q)$, and $\phi=\iota$, the natural surjection from $P_{n,2n}(\mathbb{F}_q)$ to $\GL_{n}(\mathbb{F}_q)\times \GL_{n}(\mathbb{F}_q)$. Based on the Comparison Theorem,
we immediately obtain the following result.

\begin{prop}\label{prop-unirep-to-irrdrep}
A sequence $(C_i(q))_{i\ge 0}$ of virtual unipotent representations of $\GL_n(\mathbb{F}_q)$ is inductively log-concave with respect to $(\GL_{2n}(\mathbb{F}_q), P_{n,2n}(\mathbb{F}_q),\iota)$ if and only if $(\psi(C_i(q)))_{i\ge 0}$ is inductively log-concave with respect to $(\mathfrak{S}_{2n}, \mathfrak{S}_n\times \mathfrak{S}_n, \mathrm{id})$ as a sequence of virtual representations of $\mathfrak{S}_n$.
\end{prop}

With the above proposition, the problem of proving the induced log-concavity of unipotent representations of finite general linear groups can also be transformed to that of proving the Schur positivity of certain symmetric functions.

\begin{rem}\label{rem-property}
When we consider representations of $\mathfrak{S}_{n}$, the triple $(G,H,\phi)=(\mathfrak{S}_{2n}, \mathfrak{S}_{n}\times \mathfrak{S}_{n}, \mathrm{id})$ has property ($\diamondsuit$). The underlying reason for this fact will be brought out by Proposition \ref{prop-ch}.
Furthermore, by Theorem \ref{thm-Comparison}
the triple $(G,H,\phi)=(\GL_{2n}(\mathbb{F}_q), P_{n,2n}(\mathbb{F}_q), \iota)$ also has property ($\diamondsuit$), provided that we confine ourselves to unipotent representations of $\GL_n(\mathbb{F}_q)$ instead of all of its representations, namely, we replace $\mathrm{VRep}(\GL_{n}(\mathbb{F}_q))$ with $\mathrm{VURep}(\GL_{n}(\mathbb{F}_q))$ in the condition of ($\diamondsuit$).
\end{rem}

\begin{rem} Our notion of induced log-concavity provides a way of mixing log-concavity with induction, and allows us to study certain log-concavity problems from the viewpoint of symmetric functions. Nicholas Proudfoot (private communication) pointed out that he and Weiyan Chen considered another way of mixing log-concavity with induction, with the goal of getting a statement that can be interpreted nicely in terms of symmetric functions. Let $G$ be a finite group, and let $W$ be the wreath product of $\mathfrak{S}_n$ with $G$.
Given a graded representation $V$ of $G$, there is a natural way to consider $V^{\otimes n}$ as a graded representation of $W$. They conjectured that if $V$ is $G$-equivariantly log-concave then $V^{\otimes n}$ is $W$-equivariantly log-concave. One special case of this conjecture of particular interest occurs when $G=\mathfrak{S}_k$. In this case, there is a natural embedding of $W$ into $\mathfrak{S}_{kn}$, then the representation obtained from $V^{\otimes n}$ by inducing up to $\mathfrak{S}_{kn}$ would be $\mathfrak{S}_{kn}$-equivariantly log-concave. This conjecture is appealing because it has a nice interpretation in terms of symmetric functions.
To be precise, let $f(t)$ be the polynomial in which the coefficient of $t^k$, up to sign, is equal to the Frobenius characteristic of the $k$-th graded piece of $V$. Then the corresponding polynomial for the induced representation in the previous paragraph is otained by taking the plethysm of the $n$-th complete symmetric function $h_n$ with $f(t)$.
\end{rem}

\section{Schur positivity} \label{sect-3}

In this section we will present some Schur positivity results that will be used later to prove the induced log-concavity of equivariant characteristic polynomials and Kazhdan-Lusztig polynomials of uniform matroids and $q$-niform matroids.

Let us now recall some related definitions and terminology in the theory of symmetric functions. For undefined terminology, see Stanley \cite{Stanley}.
% and Macdonald \cite{Macdonald}.
% Given a partition $\lambda$, let $\ell(\lambda)$ denote the number of its nonzero parts. Each partition $\lambda$ is
% associated to a left justified array of cells with $\lambda_i$ cells in the $i$-th row, called the Young diagram of $\lambda$.
% The cell in the $i$-th row and $j$-th column is denoted by $(i, j)$.
% The hook-length of $(i,j)$, denoted by $h_{(i,j)}$, is defined
% to be the number of cells directly to the right or directly below $(i,j)$, counting $(i,j)$ itself once.
% For example, the hook length of the cell $(1,2)$ in Young diagram of $(6,5,2)$ is $7$, as illustrated in Figure \ref{fig-1}.
% \begin{figure}[ht]
% \begin{center}
% \ytableausetup{smalltableaux}
%  \begin{ytableau}
%  {} & *(gray) & *(gray) & *(gray) & *(gray) & *(gray) \\
%  & *(gray) & & &\\
%  & *(gray) \\
% \end{ytableau}
% \end{center}
% \caption{The hook set of $(1,2)$ composed of gray cells}\label{fig-1}
% \end{figure}
Let $\lambda$ and $\mu$ are two partitions, if $\lambda/\mu$ be a skew shape, i.e., $0\leq \mu_i\le \lambda_i$ for all $i$, then the Young diagram of $\lambda/\mu$
is obtained from the Young diagram of $\lambda$ by removing the boxes in the subdiagram of $\mu$ from the top left corner.
A semistandard Young tableau $T$ of skew shape $\lambda/ \mu$ is defined to be a filling of the skew diagram $\lambda/ \mu$ with positive integers
that is weakly increasing in each row and strictly increasing in every column.
If $T$ is a semistandard Young tableau of shape $\lambda/\mu$, then we write $\mathrm{sh}(T)=\lambda/\mu$.
For instance, letting $\lambda=(6,5,4)$ and $\mu=(3,2)$, the skew diagram of $\lambda/\mu$ and a semistandard Young tableau of shape $\lambda/\mu$ are presented in Figure \ref{fig-2}.

\begin{center}
    \ytableausetup{mathmode, boxsize=0.5em, smalltableaux}
    \makeatletter\def\@captype{figure}\makeatother
    \begin{minipage}{.3\textwidth}
        \begin{ytableau}
        \none&\none  & \none  & & & \null  \\
        \none&\none & \null &  &  \\
        \null& \null &\null & \null
        \end{ytableau}
    \end{minipage}
    \makeatletter\def\@captype{figure}\makeatother
    \begin{minipage}{.3\textwidth}
        \begin{ytableau}
            \none&\none  & \none  & 3 & 4 & 6\null  \\
            \none&\none & \null 4 & 4 & 5 \\
            \null 1& \null 2&\null 7& 8\null
            \end{ytableau}
    \end{minipage}
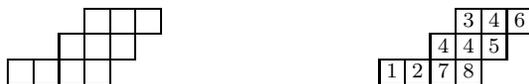
    \caption{The skew diagram and a semistandard Young tableau of shape  $(6,5,4)/(3,2)$ }\label{fig-2}
    \end{center}

%The set of all homogeneous symmetric functions of degree n
%Let $\Lambda^n$ denote the $\mathbb{Z}$-module of symmetric functions of degree $n$ in the variables
%$\mathrm{\bold x}=(x_1,x_2,\ldots)$.

The skew Schur function $s_{\lambda/\mu}$ of shape $\lambda/\mu$ in the variables $\mathrm{\bold x}$
is the formal power series
$$s_{\lambda/\mu}(\mathrm{\bold x})=\sum_{T}\mathrm{\bold x}^T,$$
where $T$ ranges over all semistandard Young tableaux of skew shape $\lambda/\mu$, and $\mathrm{\bold x}^T=x_1^{c_1}x_2^{c_2}\cdots$, if
there are $c_1$ $1$'s, $c_2$ $2$'s, etc.
If $\mu=\emptyset$, then $s_{\lambda}(\bold x)$ is called the Schur function of shape $\lambda$.
If $\lambda=(n)$, then $s_{\lambda}(\bold x)$ is called the complete symmetric function. The Jacobi-Trudi identity  \cite[Theorem 7.16.1]{Stanley}
allows us to express a skew Schur function in terms of complete symmetric functions as the following determinant
\begin{align}\label{eq-jt-identity}
s_{\lambda/\mu}(\bold x)=\det(h_{\lambda_i-\mu_j-i+j}(\bold x))_{i,j=1}^{\ell(\lambda)},
\end{align}
where $\ell(\lambda)$ denotes the length of $\lambda$
and $s_{(k)}=0$ for $k<0$.

The Schur functions play a significant role in combinatorics, representation theory and geometry.
One reason Schur functions are important is that the set $\{s_{\lambda}\mid \lambda\vdash n\}$ forms a basis for the $\mathbb{Z}$-module $\Lambda_{\mathbb{Z}}^n$ of all homogeneous symmetric functions of degree $n$ over $\mathbb{Z}$.
Another reason is that there is a canonical bijection between Schur functions
and isomorphic classes of irreducible representations of the symmetric group via the aforementioned Frobenius characteristic map $\mathrm{ch}$ in Section \ref{sect-2}. Recall that
$\mathrm{ch}\,V_{\lambda}=s_{\lambda}(\mathrm{\bold x})$ for every $\lambda\vdash n$.
Let $\Lambda_\mathbb{Z}$ denote the ring of symmetric functions over $\mathbb{Z}$, namely, %\textcolor{blue}
{$\Lambda_\mathbb{Z}=\oplus_{n\geq 0} \Lambda^n_\mathbb{Z}.$}
Define %\textcolor{blue}
{$\VRep(\mathfrak{S}_{\infty})=\oplus_{n\geq 0} \VRep(\mathfrak{S}_n)$}, which is also a graded ring with respect to the induction product ``$\ast$'' defined by \eqref{formula-induced-produce} and extended to $\VRep(\mathfrak{S}_{\infty})$ by linearity.
The Frobenius characteristic map $\mathrm{ch}$ can be extended to a linear transformation $\mathrm{ch}: \VRep(\mathfrak{S}_{\infty}) \rightarrow \Lambda_\mathbb{Z}$. Since
a representation is determined by its character, we have the following result.

\begin{prop}[{\cite[Proposition 7.18.4]{Stanley}}]\label{prop-ch}
The map $\mathrm{ch}$ is a ring isomorphism between
$\VRep(\mathfrak{S}_{\infty})$ and $\Lambda_\mathbb{Z}$. In particular, the image of the skew Specht module $V_{\lambda/\mu}$ under the map $\mathrm{ch}$ is the skew Schur function
$s_{\lambda/\mu}(\mathrm{\bold x})$.
\end{prop}

Recall that a symmetric function $f(\mathrm{\bold x})$ is said to be Schur positive, or simply $s$-positive, if it can be written as a nonnegative linear combination of Schur functions. It is known that every skew Schur function $s_{\lambda/\mu}(\mathrm{\bold x})$ is Schur positive.
In the following we will abbreviate a symmetric function $f(\mathrm{\bold x})$ to $f$ if no confusion will result.
We also use $f\ge_s g$ to represent that $f-g$ is $s$-positive.
By Proposition \ref{prop-ch}, the problem of determining whether a virtual representation $V\in \VRep(\mathfrak{S}_{\infty})$ is honest  is equivalent to proving the Schur positivity of $\ch\, V$. Recall that for any two representations $V\in\Rep(\mathfrak{S}_k),\,V'\in\Rep(\mathfrak{S}_{n-k})$ we have $\mathrm{ch}(V\ast V')=\mathrm{ch} V\cdot \mathrm{ch} V'$. Thus, we have the following result.

\begin{prop}\label{prop-il-to-sp}
A sequence
$(C_i)_{i\ge 0}$ of representation of $\mathfrak{S}_{n}$ is inductively log-concave (respectively, strongly inductively log-concave) with respect to $(\mathfrak{S}_{2n}, \mathfrak{S}_{n}\times \mathfrak{S}_{n}, \mathrm{id})$  if and only if
$(\mathrm{ch}\, C_i)^2-(\mathrm{ch}\, C_{i-1})(\mathrm{ch}\, C_{i+1})$ is Schur positive for $i\geq 1$ (respectively, $(\mathrm{ch}\, C_i)(\mathrm{ch}\, C_j)-(\mathrm{ch}\, C_{i-1})(\mathrm{ch}\, C_{j+1})$ is Schur positive for $1\leq i\leq j$).
\end{prop}

Various problems on Schur positivity have been extensively studied;
see for instance
% Bergeron and McNamara \cite{BM-2004},
 Fomin, Fulton, Li and Poon \cite{FFLP-2005}, Lam, Postnikov and Pylyavskyy \cite{LPP-2007}, Lascoux, Leclerc and Thibon \cite{LLT-1997},
%King-Welsh-Willigenburg-Stephanie \cite{KWWS-2008},
and Okounkov \cite{O-1997}.

Before presenting our results on Schur positivity, let us first recall some known results due to Lam, Postnikov and Pylyavskyy \cite{LPP-2007}. For a real number $u$, let $\lfloor u\rfloor$ be the maximal integer less than or equal to $u$ and $\lceil u\rceil$ be the minimal integer greater than or equal to $u$.
Given a positive integer $n$ and two vectors $v,w$, assume that the operations $v+w$, $\frac{v}{n}$, $\lfloor v\rfloor$ and $\lceil v\rceil$ are performed coordinate-wise.
Lam, Postnikov and Pylyavskyy \cite{LPP-2007} obtained the following result, which answered a conjecture of Okounkov \cite{O-1997}.

\begin{thm}[{\cite[Theorem 12]{LPP-2007}}]\label{thm-LPP-main}
For any two skew shapes $\lambda/\mu$ and $\nu/\rho$,
    \begin{align}
        s_{\lceil\frac{\lambda+\nu}{2}\rceil/\lceil\frac{\mu+\rho}{2}\rceil}
        s_{\lfloor\frac{\lambda+\nu}{2}\rfloor/\lfloor\frac{\mu+\rho}{2}\rfloor}
        \ge_{s}s_{\lambda/\mu}s_{\nu/\rho}.	
    \end{align}
\end{thm}

Lam, Postnikov and Pylyavskyy \cite{LPP-2007} further showed that taking conjugating partitions in the above theorem leads to a proof of a conjecture due to Fomin, Fulton, Li and Poon \cite{FFLP-2005}. Given two partitions $\lambda$ and $\mu$, let  $\lambda\cup\mu=(\alpha_1,\alpha_2,\alpha_3,\ldots)$ denote the partition obtained by rearranging all parts of $\lambda$ and $\mu$ in weakly decreasing order.
Let $\mathrm{sort}_1(\lambda,\mu):=(\alpha_1,\alpha_3,\alpha_5,\ldots)$ and
$\mathrm{sort}_2(\lambda,\mu):=(\alpha_2,\alpha_4,\alpha_6,\ldots)$.
Fomin, Fulton, Li and Poon conjectured, and Lam, Postnikov and Pylyavskyy \cite{LPP-2007} later proved the following result.

\begin{thm}[{\cite[Corollary 14]{LPP-2007}}]\label{thm-LPP-sort}
Let $\lambda/\mu $ and $\nu/ \rho$ be two skew shapes. Then $$s_{\mathrm{sort_1}(\lambda,\nu)/\mathrm{sort_1}(\mu,\rho)}s_{\mathrm{sort_2}(\lambda,\nu)/\mathrm{sort_2}(\mu,\rho)} \geq_s s_{\lambda/\mu }s_{\nu/ \rho}.$$
\end{thm}

Both Theorem \ref{thm-LPP-main} and Theorem \ref{thm-LPP-sort} will be used later but in an indirect way. Note that the Schur positivity in Theorem \ref{thm-LPP-main} and Theorem \ref{thm-LPP-sort} only depends on the values of the involved skew Schur functions, so the shapes of skew partitions are critical for the involved operations.
In order to apply Theorem \ref{thm-LPP-main} and Theorem \ref{thm-LPP-sort} to a larger context, it is natural to consider when two skew Schur functions are equal.
We would like to point out that the question of equalities among skew Schur functions has been studied by Billera, Thomas, and van Willigenburg \cite{BTW-2006}, and Reiner, Shaw and van Willigenburg \cite{RSW-2007}. For our purpose here, we need the following basic fact.

\begin{prop}[{\cite[Exercise 7.56 (a)]{Stanley}}]\label{thm-rotate-180}
Given a skew shape $\lambda/\mu$, let $(\lambda/\mu)^{r}$ denote the skew shape obtained by rotating $\lambda/\mu$
180 degrees.  Then $s_{\lambda/\mu}=s_{{(\lambda/\mu)}^{r}}.$
\end{prop}

For example, by rotating the skew diagram $\lambda/\mu=(6,5,4)/(3,2)$ 180 degrees, we obtain $(\lambda/\mu)^{r}=(6,4,3)/(2,1)$, as illustrated in Figure \ref{fig-3}.

\begin{center}
\ytableausetup{mathmode, boxsize=0.5em, smalltableaux}
\makeatletter\def\@captype{figure}\makeatother
\begin{minipage}{.3\textwidth}
\begin{ytableau}
\none&\none  & \none  & & & \null  \\
\none&\none & \null &  &  \\
\null& \null &\null & \null
\end{ytableau}
\end{minipage}
\makeatletter\def\@captype{figure}\makeatother
\begin{minipage}{.3\textwidth}
\begin{ytableau}
\none&\none  &  & & & \null  \\
\none& & \null &    \\
\null& \null &\null
\end{ytableau}
\end{minipage}
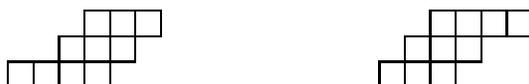
\caption{The skew diagrams $(6,5,4)/(3,2)$ and $(6,4,3)/(2,1)$}\label{fig-3}
\end{center}

Our first result on Schur positivity is concerned with
a sequence of partitions, a special case of which arises in an expression of the equivariant Kazhdan-Lusztig polynomials of uniform matroids due to Gao, Xie and Yang \cite{GXY-2021}.

\begin{thm}\label{thm-Schur-main}
    Given some fixed integers $n,k,a,b,c,d$, for any $i$, let
    $$\lambda^{(i)}/\mu^{(i)}:=(n+ki,(a+ki)^{ci+d})/((b+ki)^{ci+d}).$$
    If $i$ is an integer such that
    $\lambda^{(i-1)}/\mu^{(i-1)},\lambda^{(i)}/\mu^{(i)}$ and
    $\lambda^{(i+1)}/\mu^{(i+1)}$ are skew shapes, then
    \begin{align*}
        s^2_{\lambda^{(i)}/\mu^{(i)}}-s_{\lambda^{(i-1)}/\mu^{(i-1)}}s_{\lambda^{(i+1)}/\mu^{(i+1)}}\ge_s 0.
    \end{align*}
    % Given some fixed integers $n,k,a,b,c,d$ with $n\ge a\ge b$,
    % for $i$ satisfying $b+ki\ge 0, b+k(i-1)\ge 0, b+k(i+1)\ge 0, d+ci\ge 0, d+c(i-1)\ge 0$ and $d+c(i+1)\ge 0$,
    % let
    % $$\lambda^{(i)}/\mu^{(i)}:=(n+ki,(a+ki)^{d+ci})/((b+ki)^{d+ci}).$$
    % Then we have
    % \begin{align*}
    %     s^2_{\lambda^{(i)}/\mu^{(i)}}-s_{\lambda^{(i-1)}/\mu^{(i-1)}}s_{\lambda^{(i+1)}/\mu^{(i+1)}}\ge_s 0.
    % \end{align*}
\end{thm}

\begin{proof}
    By rotating the skew shape $\lambda^{(i)}/\mu^{(i)}$ $180$ degrees, we find that
    \begin{align*}
        (\lambda^{(i)}/\mu^{(i)})^{r}=(n+ki,(n-b)^{ci+d})/((n-a)^{ci+d}).
    \end{align*}
    In view of Theorem \ref{thm-rotate-180}, it suffices to show
    \begin{align}\label{eq-spp}
        s^2_{(\lambda^{(i)}/\mu^{(i)})^r}-s_{(\lambda^{(i-1)}/\mu^{(i-1)})^r}
        s_{(\lambda^{(i+1)}/\mu^{(i+1)})^r}\ge_s 0.
    \end{align}
%
%    \begin{align}\label{eq-spp}
%        &s^2_{(n+ki,(n-b)^{ci+d})/((n-a)^{ci+d})}\ge_s \nonumber\\
%        &s_{(n+k(i+1),(n-b)^{c(i+1)+d})/((n-a)^{c(i+1)+d})}
%        s_{(n+k(i-1),(n-b)^{c(i-1)+d})/((n-a)^{c(i-1)+d})}.
%    \end{align}
On the one hand, taking
    \begin{align*}
    \lambda &=\left(n+k(i+1),(n-b)^{c(i+1)+d}\right),\\ \mu & =((n-a)^{c(i+1)+d}),\\
    \nu &=(n+k(i-1),(n-b)^{c(i-1)+d}),\\ \rho &=((n-a)^{c(i-1)+d})
    \end{align*}
    in Theorem \ref{thm-LPP-sort}, one can verify that
    \begin{align*}
        \mathrm{sort_1}(\lambda,\nu)&=  (n+k(i+1),{(n-b)}^{ci+d}),\\
        \mathrm{sort_2}(\lambda,\nu)&=  (n+k(i-1),{(n-b)}^{ci+d}),\\
        \mathrm{sort_1}(\mu,\rho)&=\mathrm{sort_2}(\mu,\rho) =((n-a)^{ci+d}),
    \end{align*}
    and hence obtain the inequality
    \begin{align}\label{eq-spp_1}
        s_{(n+k(i+1),{(n-b)}^{ci+d})/((n-a)^{ci+d})} s_{(n+k(i-1),{(n-b)}^{ci+d})/((n-a)^{ci+d})} \ge_s s_{(\lambda^{(i-1)}/\mu^{(i-1)})^r}
        s_{(\lambda^{(i+1)}/\mu^{(i+1)})^r}.
    \end{align}
On the other hand, taking
\begin{align*}
\lambda & =(n+k(i+1),{(n-b)}^{ci+d}),\\
\mu &=((n-a)^{ci+d}),\\
\nu & =(n+k(i-1),{(n-b)}^{ci+d}),\\
\rho &=((n-a)^{ci+d})
\end{align*}
in Theorem \ref{thm-LPP-main}, one can verify that
    \begin{align*}
    \Bigl\lceil\frac{\lambda+\nu}{2}\Bigr\rceil&=\Bigl\lfloor\frac{\lambda+\nu}{2}\Bigr\rfloor
    =(n+ki,{(n-b)}^{ci+d}),\\
    \Bigl\lceil\frac{\mu+\rho}{2}\Bigr\rceil&=\Bigl\lfloor\frac{\mu+\rho}{2}\Bigr\rfloor=((n-a)^{ci+d}),
    \end{align*}
%    \begin{align*}
%    \left\lceil\frac{\lambda+\nu}{2}\right\rceil&=\left\lfloor\frac{\lambda+\nu}{2}\right\rfloor
%    =(n+ki,{(n-b)}^{ci+d}),\\
%    \left\lceil\frac{\mu+\rho}{2}\right\rceil&=\left\lfloor\frac{\mu+\rho}{2}\right\rfloor=((n-a)^{ci+d}),
%    \end{align*}
and hence obtain the inequality
    \begin{align}\label{eq-spp_2}
s^2_{(\lambda^{(i)}/\mu^{(i)})^r}
\ge_{s} s_{(n+k(i+1),{(n-b)}^{ci+d}) / ((n-a)^{ci+d})}
    s_{(n+k(i-1),{(n-b)}^{ci+d}) / ((n-a)^{ci+d})}.
\end{align}	
Combining \eqref{eq-spp_1} and \eqref{eq-spp_2} we get  \eqref{eq-spp}. This completes the proof.
\end{proof}

We proceed to state our second result on Schur positivity, which is concerned with the difference of products of Schur functions of hook shapes. To prove this result, we will use the celebrated Littlewood-Richardson rule, which we will recall below.
Given a semistandard Young tableau $T$, we say that it has type $\alpha=(\alpha_1,\alpha_2,\ldots)$, denoted $\alpha=\mathrm{type}(T)$, if $T$ has $\alpha_i$ entries equal to $i$.
The reverse reading word of $T$ is the sequence of entries of $T$ obtained by concatenating the rows of $T$ from right to left, top to bottom.
We say that a word $w_1w_2\cdots w_n$ is a lattice permutation if in any initial factor $w_1w_2\cdots w_j$,
the number of $i$'s is at least as great as the number of $(i+1)$'s for all $i$.
A Littlewood-Richardson tableau is a semistandard Young tableau $T$ such that its reverse reading word is a lattice permutation. The Littlewood-Richardson rule can be stated as follows.

\begin{thm}[{\cite[Section 7.10]{Stanley}}]\label{thm-LR}
    If
    \begin{align*}
        s_{\lambda}s_{\mu}=\sum_{\gamma}c_{\lambda\mu}^\gamma s_{\gamma},
    \end{align*}
    then the Littlewood-Richardson coefficient $c_{\lambda\mu}^\gamma$ is equal to the number of Littlewood-Richardson tableaux of shape
    $\gamma/ \mu$ and type $\lambda$.
\end{thm}

In many cases we encounter a special case of Theorem \ref{thm-LR}, called Pieri's rule, which can be stated in terms of horizontal strips as below. Recall that a horizontal strip is a skew diagram with no two squares in the same column.

\begin{thm}\cite[Theorem 7.15.7]{Stanley}
    We have
    \begin{align*}
        s_{\lambda}s_{(n)}=\sum_{\mu}s_{\mu},
    \end{align*}
    summed over all partitions $\mu$ such that $\mu/\lambda$ is a horizontal strip of size $n$.
\end{thm}

%Based on the above fact, we note that the product of two Schur functions can always be expanded as linearly nonnegative combination
%of Schur functions.
%Let $\lambda=(4,2,1)$, $\mu=(5,3,3)$ and $\gamma=(7,6,4,1)$. By using SageMath, we find that $c_{\lambda\mu}^\gamma=3$.
%In fact, the number of Littlewood-Richardson tableaux of shape $\gamma/ \mu$ and type $\lambda$ is three, as shown in figure \ref{fig-4}.
%
%\begin{center}
%    \ytableausetup{mathmode, boxsize=0.5em, smalltableaux}
%    \makeatletter\def\@captype{figure}\makeatother
%    \begin{minipage}{.3\textwidth}
%    \begin{ytableau}
%        \null  &  &  &  &  & 1 & 1 \null  \\
%        \null  &  &  & 1 & 1 & 2 \null \\
%        \null  &  &  & 2 \null\\
%        \null 3 \null
%    \end{ytableau}
%    \end{minipage}
%    \makeatletter\def\@captype{figure}\makeatother
%    \begin{minipage}{.3\textwidth}
%        \begin{ytableau}
%            \null  &  &  &  &  & 1 & 1 \null  \\
%            \null  &  &  & 1 & 2 & 2 \null \\
%            \null  &  &  & 1 \null\\
%            \null 3 \null
%        \end{ytableau}
%    \end{minipage}
%    \makeatletter\def\@captype{figure}\makeatother
%    \begin{minipage}{.3\textwidth}
%        \begin{ytableau}
%            \null  &  &  &  &  & 1 & 1 \null  \\
%            \null  &  &  & 1 & 2 & 2 \null \\
%            \null  &  &  & 3 \null\\
%            \null 1 \null
%        \end{ytableau}
%    \end{minipage}
%    \caption{Skew Littlewood-Richardson tableaux of shape $\gamma/ \mu$ and type $\lambda$}\label{fig-4}
%\end{center}

By applying the Littlewood-Richardson rule, one can expand the product of two Schur functions of hook shapes in terms of Schur functions. However, we failed to find an explicit expansion in the literature. For the sake of completeness, we state the following result and include a self-contained proof here.

\begin{lem}\label{lem-prod-hook-schur}
    For any integers $m, n\ge 1$ and $a, b\ge 0$, we have
    \begin{align}\label{eq-hook-hook}
&s_{(m,1^{a})}s_{(n,1^b)} =  \sum_{j=0}^{\min{\{m,n-1\}}} s_{(m+n-j,j+1,1^{a+b-1})} +
            \sum_{r=1}^{\min\{a,b-1\} }\sum_{j=1}^{\min\{m,n-1\} }s_{(m+n-j,j+1,2^r,1^{a+b-2r-1})}\nonumber\\[8pt]
        & + \sum_{r=0}^{\min\{a-1,b-1\} }\sum_{j=1}^{\min\{m,n\} }s_{(m+n-j+1,j+1,2^r,1^{a+b-2r-2})}+
        \sum_{j=0}^{\min{\{m-1,n-1\}}} s_{(m+n-j-1,j+1,1^{a+b})}\nonumber\\[8pt]
        & + \sum_{r=1}^{\min\{a,b\} }\sum_{j=1}^{\min\{m-1,n-1\} }s_{(m+n-j-1,j+1,2^r,1^{a+b-2r})} +
           \sum_{r=0}^{\min\{a-1,b\} }\sum_{j=1}^{\min\{m-1,n\} }s_{(m+n-j,j+1,2^r,1^{a+b-2r-1})}.
    \end{align}
\end{lem}

\begin{proof}
Assume that
    \begin{align*}
        s_{(m,1^a)}s_{(n,1^b)}=\sum_{\gamma}c_{\gamma}s_{\gamma},
    \end{align*}
where $\gamma=(\gamma_1,\gamma_2,\gamma_3,\ldots)$ ranges over all partitions of $m+n+a+b$.
We claim that if $\gamma_3\geq 3$ then $c_{\gamma}= 0$.
Actually, by Theorem \ref{thm-LR}, we know that $c_{\gamma}$ is equal to the number of Littlewood-Richardson tableaux $T$ of shape $\gamma/(n,1^b)$ and type $(m,1^a)$.
Note that when $\gamma_3\geq 3$ the diagram of $\gamma/(n,1^b)$ contains a $2\times 2$ subdiagram in the second and third rows. If there exists such a $T$, then its third row must start with a number $i\geq 2$ followed by a number $j>i$ by semistandardness and $\mathrm{type}(T)=(m,1^a)$, contradicting the lattice permutation condition of $T$.

From now on we may assume that $\gamma$ is of the form $(c,j+1,2^r,1^d)$ for some $c>j\ge 0$ and $r,d\ge 0$.
Note that each Littlewood-Richardson tableau $T$ of shape $\gamma/(n,1^b)$ and type $(m,1^a)$ will contribute a term $s_{\gamma}$ to $s_{(m,1^a)}s_{(n,1^b)}$. By the lattice permutation condition satisfied by $T$, the first row of $T$ can only be filled with $1$'s. Since $\gamma$ is of the form $(c,j+1,2^r,1^d)$, a little thought shows that if the position of $2$ in $T$ is fixed then there is a unique way to fill the boxes of $T$ with the remaining numbers $3, 4, \ldots, a+1$. In the following we shall divide all such Littlewood-Richardson tableaux into six families according to the relative position of $1$'s and $2$'s.
Since the column of $T$ is strictly increasing, there is at most one $1$ in each column of $T$. These six cases can be placed into one of two categories, say (A) and (B), the first of which consists of those tableaux without the occurrence of $1$ in the first column of $\gamma$ as shown by A.1, A.2 and A.3, and the second of which consists of those tableaux with one $1$ appearing in the first column of $\gamma$ as shown by  B.1, B.2 and B.3.
See Figure \ref{fig-5} for an illustration.

\begin{center}
    \ytableausetup{mathmode, boxsize=1.2em, nosmalltableaux}
    \makeatletter\def\@captype{figure}\makeatother
    \begin{minipage}{.34\textwidth}
    \begin{ytableau}
         *(gray) & *(gray) & *(gray) {\cdots} &  *(gray) & *(gray) & *(gray) & \textcolor{blue}{1} &\textcolor{blue}{\cdots}& \textcolor{blue}{1}  \\
         *(gray) & \textcolor{blue}{1} &\textcolor{blue}{\cdots}& \textcolor{blue}{1}  \null \\
         *(gray)  \\
         *(gray)\vdots \\
         *(gray) \\
         \textcolor{red}{{\bf 2}} \\
         {3} \\
         {\vdots}\\
         \scriptstyle a+1
    \end{ytableau}
    \\ \centering A.1
    \end{minipage}
    \makeatletter\def\@captype{figure}\makeatother
    \begin{minipage}{.34\textwidth}
        \begin{ytableau}
            *(gray) & *(gray) & *(gray){\cdots} & *(gray) & *(gray) & *(gray) & \textcolor{blue}{1} &\textcolor{blue}{\cdots}& \textcolor{blue}{1}  \\
            *(gray) & \textcolor{blue}{1} &\textcolor{blue}{\cdots}& \textcolor{blue}{1}  \null \\
            *(gray)  & \textcolor{red}{{\bf 2}} \\
            % *(gray) & \textcolor{blue}{3}\\
            *(gray)\vdots & \vdots \\
            *(gray) & \scriptstyle{r+1} \\
            *(gray) \\
            \scriptstyle{r+2}\\
            \vdots\\
            {{\scriptstyle a+1}} \\
       \end{ytableau}
       \\ \centering A.2
    \end{minipage}
    \makeatletter\def\@captype{figure}\makeatother
    \begin{minipage}{.3\textwidth}
        \begin{ytableau}
            *(gray) & *(gray) & *(gray)\cdots &  *(gray) & *(gray) & *(gray) & \textcolor{blue}{1} &\textcolor{blue}{\cdots}& \textcolor{blue}{1}  \\
            *(gray) & \textcolor{blue}{1} &\textcolor{blue}{\cdots}& \textcolor{blue}{1}& \textcolor{red}{{\bf 2}} \null \\
            *(gray) &  {3}  \\
            *(gray) \vdots &  {\vdots} \\
            *(gray) & \scriptstyle {r+2}\\
            *(gray) \\
            \scriptstyle {r+3}\\
            {\vdots}\\
            \scriptstyle {a+1}\\
       \end{ytableau}
       \\ \centering A.3
    \end{minipage}
\end{center}

\begin{center}
    \ytableausetup{mathmode, boxsize=1.2em, nosmalltableaux}
    \makeatletter\def\@captype{figure}\makeatother
    \begin{minipage}{.34\textwidth}
    \begin{ytableau}
         *(gray) & *(gray) & *(gray)\cdots &  *(gray) & *(gray) & *(gray) & \textcolor{blue}{1} &\textcolor{blue}{\cdots}& \textcolor{blue}{1}  \\
         *(gray) & \textcolor{blue}{1} &\textcolor{blue}{\cdots}& \textcolor{blue}{1}  \null \\
         *(gray)  \\
         *(gray) \vdots \\
         *(gray) \\
         \textcolor{blue}{{\bf 1}} \\
         \textcolor{red}{{\bf 2}} \\
         3\\
         \vdots\\
         \scriptstyle {a+1}\\
    \end{ytableau}
    \\ \centering B.1
    \end{minipage}
    \makeatletter\def\@captype{figure}\makeatother
    \begin{minipage}{.34\textwidth}
        \begin{ytableau}
            *(gray) & *(gray) & *(gray)\cdots & *(gray) & *(gray) & *(gray) & \textcolor{blue}{1} &\textcolor{blue}{\cdots}& \textcolor{blue}{1}  \\
            *(gray) & \textcolor{blue}{1} &\textcolor{blue}{\cdots}& \textcolor{blue}{1}  \null \\
            *(gray)  & \textcolor{red}{{\bf 2}} \\
            *(gray) \vdots & \vdots \\
            *(gray) & \scriptstyle {r+1}\\
            *(gray) \\
            \textcolor{blue}{{\bf 1}} \\
            \scriptstyle {r+2}\\
            {\vdots}\\
            \scriptstyle {a+1}\\
       \end{ytableau}
       \\ \centering B.2
    \end{minipage}
    \makeatletter\def\@captype{figure}\makeatother
    \begin{minipage}{.3\textwidth}
        \begin{ytableau}
            *(gray) & *(gray) & *(gray) \cdots &  *(gray) & *(gray) & *(gray) & \textcolor{blue}{1} &\textcolor{blue}{\cdots}& \textcolor{blue}{1}  \\
            *(gray) & \textcolor{blue}{1} &\textcolor{blue}{\cdots}& \textcolor{blue}{1}& \textcolor{red}{{\bf 2}} \null \\
            *(gray) &  \textcolor{blue}{3}  \\
            *(gray) \vdots & \vdots \\
            *(gray) &\scriptstyle {r+2} \\
            *(gray) \\
            \textcolor{blue}{{\bf 1}} \\
            \scriptstyle {r+3}\\
            {\vdots}\\
            \scriptstyle {a+1}\\
       \end{ytableau}
       \\ \centering B.3
    \end{minipage}
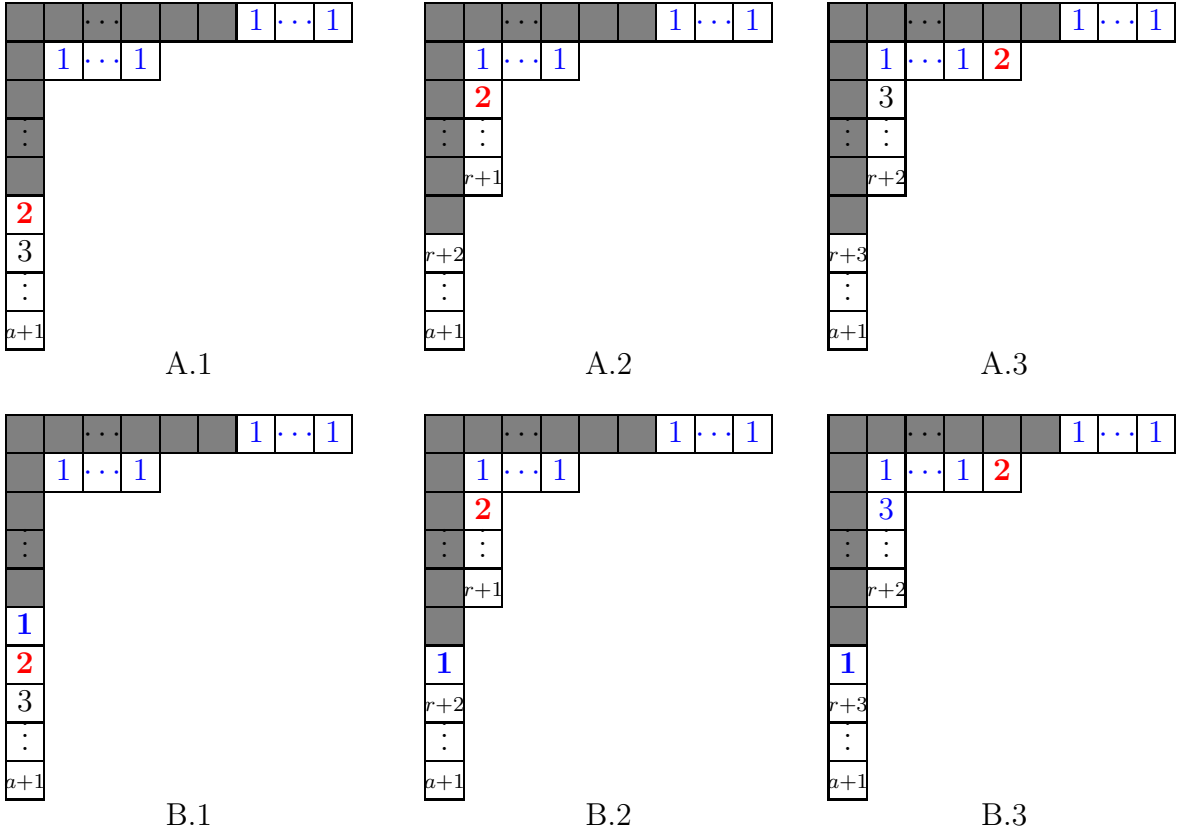
    \caption{Littlewood-Richardson tableaux of shape $(c,j+1,2^r,1^d)/(n,1^b)$ and type $(m,1^a)$}\label{fig-5}
\end{center}

The aforementioned cases are listed as follows.

\begin{itemize}
\item[(\textbf{A})] $T$ has no occurrence of $1$ in the first column of $\gamma$.
\begin{itemize}
    \item[(\textbf{A.1})] $T$ has $j$ occurrences of $1$ in the second row of $\gamma$ and one occurrence of $2$ in the first column. In this case, we have $0\le j\le \min{\{m,n-1\}}$ and    $\gamma=(m+n-j,j+1,1^{a+b-1})$, corresponding to the first summation of \eqref{eq-hook-hook}.

    \item[(\textbf{A.2})]  $T$ has $j$ occurrences of $1$ in the second row of $\gamma$, and the number $2$ lies in its second column and third row.
        For this case, we have $1\le j \le \min\{m,n-1\}$. Moreover, if $T$ contains entries $3, 4, \ldots,r+1$ in its second column, then $1\le r \le \min\{a,b-1\}$ and     $\gamma=(m+n-j,j+1,2^r,1^{a+b-2r-1})$. This case contributes the second summation of \eqref{eq-hook-hook} to the Schur expansion of $s_{(m,1^a)}s_{(n,1^b)}$.

    \item[(\textbf{A.3})] $T$ has $j-1$ occurrences of $1$ and one occurrence of $2$ in the second row of $\gamma$. In this case, we have $1\le j \le \min\{m,n\}$ since there must exist $1$ in the first row of $T$ by the lattice permutation condition. Moreover, if $T$ contains entries $3, 4, \ldots,r+2$ in its second column, then $0\le r \le \min\{a-1,b-1\}$ and $\gamma=(m+n-j+1,j+1,2^r,1^{a+b-2r-2})$. This case contributes the third summation of \eqref{eq-hook-hook} to the Schur expansion of $s_{(m,1^a)}s_{(n,1^b)}$.
\end{itemize}

\item[(\textbf{B})] $T$ has an entry $1$ in the first column of $\gamma$.
    \begin{itemize}
    \item[(\textbf{B.1})] Similar to Case A.1. $T$ also has $j$ occurrences of $1$ in the second row of $\gamma$ and one occurrence of $2$ in the first column. In this case, we have $0\le j\le \min{\{m-1,n-1\}}$ and $\gamma=(m+n-j-1,j+1,1^{a+b})$, corresponding to the fourth summation of \eqref{eq-hook-hook}.

    \item[(\textbf{B.2})] Similar to Case A.2.
    $T$ has $j$ occurrences of $1$ in the second row of $\gamma$, and the number $2$ lies in its second column and third row.
        For this case, we have $1\le j \le \min\{m-1,n-1\}$. Moreover, if $T$ contains entries $3, 4, \ldots,r+1$ in its second column, then $1\le r \le \min\{a,b\}$ and $\gamma=(m+n-j-1,j+1,2^r,1^{a+b-2r})$. This case contributes the fifth summation of \eqref{eq-hook-hook} to the Schur expansion of $s_{(m,1^a)}s_{(n,1^b)}$.

    \item[(\textbf{B.3})] Similar to Case A.3. $T$ also has $j-1$ occurrences of $1$ and one occurrence of $2$ in the second row of $\gamma$. In this case, we have $1\le j \le  \min\{m-1,n\}$. Moreover, if $T$ contains entries $3, 4, \ldots,r+2$ in its second column, then $0\le r \le \min\{a-1,b\}$ and $\gamma=(m+n-j,j+1,2^r,1^{a+b-2r-1})$. This case corresponds to the sixth summation of \eqref{eq-hook-hook}.
\end{itemize}
\end{itemize}
Combining these six cases, we complete the proof.
\end{proof}

Based on the above lemma, we obtain the following Schur positivity result on the differences of products of two Schur functions of hook shapes.

\begin{cor}\label{cor-sp-hook-schur}
For integers $a,b,p,q$ satisfying $1\le a\le b$ and $ 1\le p\le q$, we have
        \begin{align}\label{eq-spos-hook-1}
            s_{(q,1^{a})}s_{(p,1^{b})}\ge_s s_{(q+1,1^{a-1})}s_{(p-1,1^{b+1})}.
        \end{align}
%        In particular, taking $q=k-a,p=k-b$ with $k>b$ leads to
%        \begin{align}\label{eq-spos-hook-2}
%            s_{(k-a,1^{a})}s_{(k-b,1^{b})}\ge_s   s_{(k-a+1,1^{a-1})}s_{(k-b-1,1^{b+1})}.
%        \end{align}
    \end{cor}

    \begin{proof}
        We use \eqref{eq-hook-hook} to expand $s_{(q,1^{a})}s_{(p,1^{b})}$ and
        $s_{(q+1,1^{a-1})}s_{(p-1,1^{b+1})}$ explicitly.
        Since
        \begin{align*}
            q+p=(q+1)+(p-1)\qquad a+b=(a-1)+(b+1),
        \end{align*}
        the $l$-th ($1\leq l\leq 6$) summation in the Schur expansion of $s_{(q,1^{a})}s_{(p,1^{b})}$ has the same summand as that  of $s_{(q+1,1^{a-1})}s_{(p-1,1^{b+1})}$.
        Moreover, the number of terms in the former summation is at least as great as the number of terms in the former summation, since
           \begin{align*}
             \min{\{q,p-1\}}=p-1   &> p-2=\min{\{q+1,p-2\}}\\
             \min{\{q-1,p-1\}}=p-1 &> p-2=\min{\{q,p-2\}} \\
             \min{\{q,p\}}=p       &> p-1=\min{\{q+1,p-1\}}\\
             \min{\{q-1,p\}}=
               \left\{\begin{array}{ll}
                      p-1,& \mbox{ if } q=p\\
                       p, & \mbox{ otherwise}
                      \end{array}\right.     &\ge p-1=\min{\{q,p-1\}}\\
             \min{\{a,b-1\}}=
               \left\{\begin{array}{ll}
                      a-1,& \mbox{ if }b=a\\
                       a, & \mbox{ otherwise}
               \end{array}\right.            &\ge a-1=\min{\{a-1,b\}}\\
             \min{\{a-1,b-1\}}=a-1   &> a-2=\min{\{a-2,b\}}\\
             \min{\{a,b\}}=a         &> a-1=\min{\{a-1,b+1\}}\\
             \min{\{a-1,b\}}=a-1     &> a-2=\min{\{a-2,b+1\}}.
         \end{align*}
    This completes the proof.
    \end{proof}

Bergeron, Biagioli, and Rosas \cite[Theorem 3.1]{BBR-2005} also obtained some Schur positivity results on the differences of products of Schur functions of hook shapes, thus providing some positive evidence for a conjecture proposed by Fomin, Fulton, Li, and Poon \cite{FFLP-2005} on the $\ast$-operation of ordered pairs of partitions.
Given an ordered pair $(\mu,\nu)$ of partitions with the same number of parts, define a new ordered pair $(\mu,\nu)^*=(\lambda(\mu,\nu), \rho(\mu,\nu))$ by
\begin{align*}
    \lambda_k= &\mu_k-k+\#\{j\,|\,\nu_j-j\ge \mu_k-k\}\\
    \rho_j= &\nu_j-j+1+\#\{k\,|\,\mu_k-k>\nu_j-j \}.
\end{align*}
Fomin, Fulton, Li, and Poon conjectured that the difference $s_{\lambda(\mu,\nu)}s_{\rho(\mu,\nu)}-s_{\mu}s_{\nu}$ is $s$-positive for any ordered pair $(\lambda,\mu)$. Bergeron, Biagioli, and Rosas confirmed this conjecture for many infinite families {\cite[Theorem 3.1]{BBR-2005}},
including pairs of hook shapes. By the definition of $*$-operation it is routine to verify that, for $m,n\geq 1$,
    \begin{align*}
        \left((m,1^i),(n,1^j)\right)^{*}=
        \begin{cases}
            \left((m,1^{j-1}),(n,1^{i+1})\right),\qquad \mbox{ if } \,m\le n \mbox{ and } 0\le i< j\\[5pt]
            \left((m,1^{j}),(n,1^{i})\right),\qquad \mbox{ if }\,m\le n \mbox{ and }  i\ge j\ge 0\\[5pt]
            \left((m-1,1^{j-1}),(n+1,1^{i+1})\right),\qquad \mbox{ if }\,m>n \mbox{ and }  0\ge i<j\\[5pt]
            \left((m-1,1^{j}),(n+1,1^{i})\right),\qquad \mbox{ if }\,m> n\quad i\ge j\ge 0.
        \end{cases}
    \end{align*}
Bergeron, Biagioli, and Rosas obtained the following result.

\begin{thm}[{\cite[Theorem 3.1]{BBR-2005}}]\label{thm-hook-hook}
\begin{itemize}
\item[(i)] If $1\le m\le n$ and $0\leq i< j$, then $s_{(m,1^i)}s_{(n,1^j)}\le_s s_{(m,1^{j-1})}s_{(n,1^{i+1})}$.
\item[(ii)] If $1\leq m\le n$ and  $i\ge j\geq 0$,  then $s_{(m,1^i)}s_{(n,1^j)}\le_s s_{(m,1^{j})}s_{(n,1^{i})}$.

\item[(iii)] If $m>n\geq 1$  and $0\leq i< j$,  then  $s_{(m,1^i)}s_{(n,1^j)}\le_s s_{(m-1,1^{j-1})}s_{(n+1,1^{i+1})}$.

\item[(iv)]  If $m>n\geq 1$  and $i\ge j\geq 0$, then  $s_{(m,1^i)}s_{(n,1^j)}\le_s s_{(m-1,1^{j})}s_{(n+1,1^{i})}$.
\end{itemize}
\end{thm}

We would like to point out Lemma \ref{lem-prod-hook-schur} allows us to directly verify the above theorem.
It is natural to consider whether Theorem \ref{thm-hook-hook} can be used
to prove Corollary \ref{cor-sp-hook-schur}.
In fact, it is also possible to use Theorem \ref{thm-hook-hook} to derive \eqref{eq-spos-hook-1}, but in an indirect way.

            % \begin{align}
            %     s_{(q+1,1^{a-1})}s_{(p-1,1^{b+1})}\le_s s_{(q,1^{a})}s_{(p,1^{b})}\mbox{ for }
            %     1\le a\le b\mbox{ and } 1\le p\le q
            % \end{align}
\noindent \textit{The second proof of Corollary \ref{cor-sp-hook-schur}.}
Note that $1\le a\le b$ and $1\le p\le q$. First, by using (ii) of Theorem \ref{thm-hook-hook}, we get
$$s_{(p-1,1^{b+1})}s_{(q+1,1^{a-1})}
                \le_s s_{(p-1,1^{a-1})}s_{(q+1,1^{b+1})}.$$
Secondly, by applying (iv) of Theorem \ref{thm-hook-hook}, we obtain
$$
s_{(q+1,1^{b+1})}s_{(p-1,1^{a-1})}
                \le_s s_{(q,1^{a-1})}s_{(p,1^{b+1})}.
$$
Again, by using (ii) of Theorem \ref{thm-hook-hook}, we find that
$$
s_{(p,1^{b+1})}s_{(q,1^{a-1})}
                \le_s s_{(p,1^{a-1})}s_{(q,1^{b+1})}.$$
Finally, by (i) of Theorem \ref{thm-hook-hook} we get
$$s_{(p,1^{a-1})}s_{(q,1^{b+1})} \le_s s_{(p,1^{b})}s_{(q,1^{a})}.$$
Combining the above four inequalities leads to
$$s_{(q+1,1^{a-1})}s_{(p-1,1^{b+1})}\le_s s_{(q,1^{a})}s_{(p,1^{b})},$$
as desired.
\qed

\section{Invariants of $q$-niform and uniform matroids}
\label{sect-4}

The main objective of this section is to study the induced log-concavity of some polynomial invariants associated with equivariant $q$-niform and uniform matroids, including equivariant characteristic polynomials, equivariant Kazhdan-Lusztig polynomials, equivariant inverse Kazhdan-Lusztig polynomials, and equivariant $Z$-polynomials.

Before our detailed investigation of the log-concavity of these polynomials, let us first recall some terminology and notation in matroid theory. {For more information on matroids, see Oxley \cite{Oxley-2011}.}
 Let $M=(E,\mathcal{F})$ be a matroid with ground set $E$ and the set of flats $\mathcal{F}$.
 We write $\mathcal{L}(M)$ for the lattice of flats of $M$.
 For any flat $F$ of $M$, let $M^F$ denote the localization of $M$ at $F$ with ground set $F$, and
 each of its flats is the intersection of $F$ and a flat
 of $M$. Dually, let $M_F$ be the contraction of $M$ at $F$, and this is a matroid on the ground set $E\setminus F$
 with each of its flats being $G\setminus F$ for some flat $G$ containing $F$ in $M$.
 For any subset $I$ of $E$, let $\rk I$ denote the rank of $I$ in the matroid $M$,
 and let $\rk M=\rk E$.
For any matroid $M=(E,\mathcal{F})$, let $W$ be a finite group acting on $E$ and preserving $M$.
Such a matroid $M$ equipped with the action of $W$ is called an equivariant matroid, denoted by $W\curvearrowright M$.

We proceed to recall the definition of uniform matroids and $q$-niform matroids. For any integers $m\ge 0$ and $d\ge 1$, let $U_{m,d}$ denote the uniform matroid of rank $d$ on ground set $\{1,2,\ldots,m+d\}$.
The uniform matroid $U_{m,d}$ naturally admits the action of the symmetric group $\mathfrak{S}_{m+d}$.
For any prime power $q$, a $q$-analogue of $U_{m,d}$, denoted by $U_{m,d}(q)$,  was studied in \cite{HRS-2021}, whose bases are linearly independent subsets of $\mathbb{F}^{m+d}_q$ of size $d$. The matroid $U_{m,d}(q)$ was called $q$-niform matroid by Proudfoot \cite{P-2019}.
Proudfoot noted that the $q$-niform matroid $U_{m,d}(q)$ naturally admits the action of the finite general linear group $\GL_{m+d}(\mathbb{F}_q)$. This section is mainly concerned with the equivariant matroids $\mathfrak{S}_{m+d} \curvearrowright  U_{m,d}$ and $\GL_{m+d}(\mathbb{F}_q) \curvearrowright U_{m,d}(q)$.

In the following we will successively study the induced log-concavity of equivariant characteristic polynomials, equivariant Kazhdan-Lusztig polynomials, equivariant inverse Kazhdan-Lusztig polynomials, and equivariant $Z$-polynomials of $\mathfrak{S}_{m+d} \curvearrowright  U_{m,d}$ and $\GL_{m+d}(\mathbb{F}_q) \curvearrowright U_{m,d}(q)$. Throughout this section, we always mean the induced log-concavity of representations of $\mathfrak{S}_{m+d}$ and $\GL_{m+d}(\mathbb{F}_q)$ in the sense of Section \ref{sect-2}, and simply say that these polynomials are inductively log-concave or strongly inductively log-concave. Our general strategy is to use Proposition \ref{prop-unirep-to-irrdrep} to transform problems on polynomial invariants of $\GL_{m+d}(\mathbb{F}_q) \curvearrowright U_{m,d}(q)$ to problems on the corresponding counterparts of $\mathfrak{S}_{m+d} \curvearrowright  U_{m,d}$, and further to use the Frobenius characteristic map to transform to problems on certain symmetric functions.
%
%all polynomial invariants of $\GL_{m+d}(\mathbb{F}_q) \curvearrowright U_{m,d}(q)$ we considered here involve only unipotent representations of finite general linear groups
%
%In this section, we will give the defintion of equivariant reduced characteristic polynomials of equivariant matroids, inspired by the reduced characteristic polynomials. Then we will compute the expression of equivariant reduced characteristic polynomials of uniform matroids equipped with the action of symmetric groups and that of $q$-niform matroid equipped with the action of finite general linear groups.
%
%Later, we verified the strongly inductively log-concavity of equivariant reduced characteristic polynomials of $q$-niform matroid equipped with the action of finite general linear groups by the results of Schur positivity, which implies the stronger inductively log-concavity of equivariant characteristic polynomials of the same matroid.
%
%Finally, for any $q$-niform matroid equipped with the action of finite general linear groups, we proceed to prove the inductive inductively log-concavity of its equivariant Kazhdan-Lusztig polynomials, its equivariant inverse  Kazhdan-Lusztig polynomials, and partical results for its equivariant $Z$-polynomials.

\subsection{Equivariant characteristic polynomials}\label{subsect-4.1}

In this subsection, we aim to prove the strongly induced log-concavity of equivariant characteristic polynomials of
$\mathfrak{S}_{m+d} \curvearrowright  U_{m,d}$ and $\GL_{m+d}(\mathbb{F}_q) \curvearrowright U_{m,d}(q)$.

Given a matroid $M$ with ground set $E$, recall that its characteristic polynomial $\chi_M(t)$ and reduced characteristic polynomial $\overline{\chi}_M(t)$ of $M$ can be defined as
\begin{align}\label{eq-char-rchar-pol}
    \chi_M(t)=\sum_{I\subseteq E}(-1)^{|I|}t^{\rk M-\rk I},\qquad
    ({t-1})\overline{\chi}_M(t)={\chi_M(t)}.
\end{align}
For an equivariant matroid $W\curvearrowright M$,
Gedeon, Proudfoot and Young  \cite{GPY-2017} introduced the concept of the equivariant characteristic polynomial, denoted by $H_M^W(t)$. This polynomial is given by
\begin{align*}
    H_M^W(t):=\sum_{i=0}^{\rk M}(-1)^i t^{\rk M-i}\OS_{M,i}^W\in \VRep(W)[t],
\end{align*}
where $\OS_{M,i}^W\in\Rep(W)$ is the degree $i$ part of the Orlik-Solomon algebra of $M$. For the definition of Orlik-Solomon algebra of a matroid, see \cite[Definition 3.1]{EF-1999}.
Gedeon, Proudfoot and Young  \cite[Lemma 2.1]{GPY-2017} showed that $H_M^W(\tau)=0$, an equivariant version of the relation $\chi_M(1)=0$. Inspired by \eqref{eq-char-rchar-pol}, it is natural to introduce the equivariant reduced characteristic polynomial $\widetilde{H}^W_M(t)$ of $W\curvearrowright  M$ by
    \begin{align}\label{eq-equivariant-red-char-pol}
        (t-\tau)\widetilde{H}^W_M(t)=H_M^W(t).
    \end{align}

%\begin{lem}
%    $\widetilde{H}_M^W(t)$
%\end{lem}
%\begin{proof}
%    Suppose that $\rk M=d$, let
%    \begin{align*}
%        \widetilde{H}_M^W(t) =\sum_{i=1}^{d}(-1)^{i-1}a_it^{d-i}.
%    \end{align*}
%    Then
%    \begin{align*}
%        \sum_{i=1}^{d}(-1)^{i-1}a_it^{d-i}(t-\tau)
%        &=\sum_{i=1}^{d}(-1)^{i-1}a_it^{d-(i-1)}-\sum_{i=1}^{d}(-1)^{i-1}a_it^{d-i}\\
%        &=\sum_{i=0}^{d-1}(-1)^{i}a_{i+1}t^{d-i}-\sum_{i=1}^{d}(-1)^{i-1}a_it^{d-i}\\
%        &=a_1t^d + \sum_{i=1}^{d-1}(-1)^{i}(a_{i+1}+a_i)t^{d-i} +(-1)^{d}a_d\\
%        &=\OS_{M,0}^Wt^d+ \sum_{i=1}^{d-1}(-1)^i \OS_{M,i}^Wt^{d-i}+(-1)^d\OS_{M,d}^W.
%    \end{align*}
%    Then
%    \begin{align*}
%        a_1=\OS_{M,0}^W\qquad a_{i+1}+a_i=\OS_{M,i}^W\qquad a_d=\OS_{M,d}^W
%    \end{align*}
%    then
%    \begin{align*}
%        a_1 &=\OS_{M,0}^W\\
%         a_2&=\OS_{M,1}^W-\OS_{M,0}^W\\
%         a_3&=\OS_{M,2}^W-\OS_{M,1}^W+\OS_{M,0}^W\\
%         a_4&=\OS_{M,3}^W-\OS_{M,2}^W+\OS_{M,1}^W-\OS_{M,0}^W\\
%         a_5&=\OS_{M,4}^W-\OS_{M,3}^W+\OS_{M,2}^W-\OS_{M,1}^W+\OS_{M,0}^W\\
%         a_i&=\sum_{j=0}^{i-1} (-1)^{i-j-1} \OS_{M,j}^W\\
%         a_d&=\sum_{j=0}^{d-1}\OS_{M,j}^W=\OS_{M,d}^W??
%    \end{align*}
%\end{proof}
Gedeon, Proudfoot, and Young \cite[Proposition 3.9]{GPY-2017} proved that the equivariant characteristic polynomial of $\mathfrak{S}_{m+d} \curvearrowright U_{m,d}$ is
\begin{align}\label{eq-char-pol-umd}
    H_{U_{m,d}}^{\mathfrak{S}_{m+d}}(t)
    & =\sum_{i=0}^{d-1}(-1)^i\left(V_{(m+d-i,1^i)}+V_{(m+d-i+1,1^{i-1})}\right)t^{d-i}+(-1)^dV_{(m+1,1^{d-1})}.
\end{align}
Based on the Comparison theorem, Proudfoot \cite{P-2019} noted that the equivariant characteristic polynomial of $GL_{m+d}(q)\curvearrowright U_{m,d}(q)$ is given by
\begin{align}\label{eq-char-pol-qniform}
    H_{U_{m,d}(q)}^{\GL_{m+d}(\mathbb{F}_q)}(t)
    & =\sum_{i=0}^{d-1}(-1)^i\left(V_{(m+d-i,1^i)}(q)+V_{(m+d-i+1,1^{i-1})}(q)\right)t^{d-i}
    +(-1)^dV_{(m+1,1^{d-1})}(q).
\end{align}
By comparing coefficients, from \eqref{eq-equivariant-red-char-pol}, \eqref{eq-char-pol-umd} and \eqref{eq-char-pol-qniform} it follows that the equivariant reduced characteristic polynomials of $\mathfrak{S}_{m+d} \curvearrowright U_{m,d}$
    and $\GL_{m+d}(\mathbb{F}_q)\curvearrowright U_{m,d}(q)$ are given by
    \begin{align*}
        \widetilde{H}_{U_{m,d}}^{\mathfrak{S}_{m+d}}(t)
        & =\sum_{i=1}^{d}(-1)^{i-1}V_{(m+d-i+1,1^{i-1})}t^{d-i},\\%\label{eq-red-char-pol-umd}\\
        \widetilde{H}_{U_{m,d}(q)}^{\GL_{m+d}(\mathbb{F}_q)}(t)
        & =\sum_{i=1}^{d}(-1)^{i-1}V_{(m+d-i+1,1^{i-1})}(q)t^{d-i}.%\label{eq-red-char-pol-qumd}
    \end{align*}
%
%\begin{proof}
%    The proof of \eqref{eq-char-pol-qumd} is similar to the proof of \eqref{eq-char-pol-umd}, we omit it.
%    Suppose that $\widetilde{H}_{U_{m,d}}^{\mathfrak{S}_{m+d}}(t)=\sum_{i=1}^da_it^{{d-i}}$, since
%    \begin{align*}
%       H_{U_{m,d}}^{\mathfrak{S}_{m+d}}(t)=\widetilde{H}_{U_{m,d}}^{\mathfrak{S}_{m+d}}(t)\cdot(t-\tau),
%    \end{align*}
%    then
%    \begin{align*}
%       H_{U_{m,d}}^{\mathfrak{S}_{m+d}}(t)
%       &=\sum_{i=1}^da_it^{{d-i}}\cdot (t-\tau)\\
%    %    &=\sum_{i=1}^da_it^{{d-i+1}}-\sum_{i=1}^da_it^{{d-i}}\\
%    %    &=\sum_{i=0}^{d-1}a_{i+1}t^{{d-i}}-\sum_{i=1}^da_it^{{d-i}}\\
%       &=\sum_{i=1}^{d-1}(a_{i+1}-a_i)t^{{d-i}}+a_1t^{d}-a_d.
%    \end{align*}
%    On the other hand,
%    \begin{align*}
%        H_{U_{m,d}}^{\mathfrak{S}_{m+d}}(t)=\sum_{i=1}^{d-1}(-1)^i\left(V_{(m+d-i,1^i)}+V_{(m+d-i+1,1^{i-1})}\right)t^{d-i} +V_{(m+d)}t^d +(-1)^dV_{(m+1,1^{d-1})}.
%    \end{align*}
%    Then, we find
%    \begin{align*}
%       a_1=V_{(m+d)},\qquad a_d=(-1)^{d-1}V_{(m+1,1^{d-1})}
%    \end{align*}
%    and
%    \begin{align*}
%        a_{i+1}-a_i=(-1)^iV_{(m+d-i,1^i)}-(-1)^{i-1}V_{(m+d-i+1,1^{i-1})}.
%    \end{align*}
%    Therefore, for $1\le i\le d$, we have $a_i=(-1)^{i-1}V_{(m+d-i+1,1^{i-1})}$.
%\end{proof}

By abuse of notation, we say that these (reduced) characteristic polynomials are induced log-concave or strongly inductively log-concave, if their corresponding coefficient sequences without signs are induced log-concave or strongly inductively log-concave.
We have the following result.

\begin{thm}\label{thm-charpol-indlog-q}
    For any $d\ge 1$ and $m\ge 0$, the polynomial $\widetilde{H}_{U_{m,d}(q)}^{\GL_{m+d}(\mathbb{F}_q)}(t)$ is strongly inductively log-concave, and so is $\widetilde{H}_{U_{m,d}}^{\mathfrak{S}_{m+d}}(t)$.
\end{thm}

\begin{proof}
By Proposition \ref{prop-unirep-to-irrdrep}, it is sufficient to prove the strongly induced log-concavity of $\widetilde{H}_{U_{m,d}}^{\mathfrak{S}_{m+d}}(t)$.
Since the image of $\widetilde{H}_{U_{m,d}}^{\mathfrak{S}_{m+d}}(t)$ under the Frobenius map is
    \begin{align*}
        \ch\,\widetilde{H}_{U_{m,d}}^{\mathfrak{S}_{m+d}}(t)
        =\sum_{i=1}^{d}(-1)^{i-1}s_{(m+d-i+1,1^{i-1})}t^{d-i},
    \end{align*}
    by Proposition \ref{prop-il-to-sp}
    it suffices to show that for any $1\leq i\leq j\le d-2$
    \begin{align*}
        s_{(m+d-i,1^{i})}s_{(m+d-j,1^{j})}\ge_s s_{(m+d-i+1,1^{i-1})}s_{(m+d-j-1,1^{j+1})}.
    \end{align*}
    Taking $a=i,b=j$ and $q=m+d-i,p=m+d-j$ in Corollary \ref{cor-sp-hook-schur} leads to the desired Schur positivity. This completes the proof.
\end{proof}

Now we come to the main result of this section.

    \begin{thm}\label{thm-chpol-sindlog-q}
        For any $d\ge 1$ and $m\ge 0$, both ${H}_{U_{m,d}(q)}^{\GL_{m+d}(\mathbb{F}_q)}(t)$ and ${H}_{U_{m,d}}^{\mathfrak{S}_{m+d}}(t)$ are strongly inductively log-concave.
    \end{thm}

\begin{proof}
Note that
\begin{align*}
{H}_{U_{m,d}}^{\mathfrak{S}_{m+d}}(-t)=-(t+\tau)\widetilde{H}_{U_{m,d}}^{\mathfrak{S}_{m+d}}(-t).
\end{align*}
Similarly,
\begin{align*}
{H}_{U_{m,d}(q)}^{\GL_{m+d}(\mathbb{F}_q)}(-t)=-(t+\tau)
\widetilde{H}_{U_{m,d}(q)}^{\GL_{m+d}(\mathbb{F}_q)}(-t).
\end{align*}
By combining Theorem \ref{thm-charpol-indlog-q} and Proposition \ref{prop-times (t+1)} (together with Remark \ref{rem-property}), we complete the proof.
\end{proof}

As an immediate consequence of the above result, we obtain the log-concavity of $\chi_{U_{m,d}}$ and $\chi_{U_{m,d}(q)}$. Note that the log-concavity of $\chi_M(t)$ for any matroid $M$ was proved by Adiprasito, Huh, and Katz \cite{AHK-2018}.

\subsection{Equivariant Kazhdan-Lusztig polynomials}\label{subsect-4.2}
The aim of this subsection is to prove the inductively log-concavity of equivariant Kazhdan-Lusztig polynomials of $\mathfrak{S}_{m+d} \curvearrowright U_{m,d}$
and $\GL_{m+d}(\mathbb{F}_q)\curvearrowright U_{m,d}(q)$, which implies the log-concavity of the corresponding Kazhdan-Lusztig polynomials.

First let us recall the definition of matroid Kazhdan-Lusztig polynomial.
Elias, Proudfoot and Wakefield \cite[Theorem 2.2]{EPW-2016} showed that
there exists a unique way to assign each matroid $M$ a polynomial $P_M(t)\in\mathbb{Z}[t]$ subject to the following three conditions:
    \begin{itemize}
        \item If the ground set $E$ of $M$ is empty, then $P_M(t)=1$;
        \item If $\rk M>0$, then the degree of $P_M(t)$ is strictly less than $\rk M/2$;
        \item For every matroid $M$, we have
        $$t^{\rk M}P_M(t^{-1})=\sum_{F\in\mathcal{L}(M)}\chi_{M^F}(t)P_{M_F}(t).$$
    \end{itemize}
The above polynomial $P_M(t)$ is called the Kazhdan-Lusztig polynomial of $M$.
As pointed out in \cite[Remark 2.2]{KNPV-2022}, with the above definition of
$P_M(t)$ one can show that $P_M(t)=0$ if $M$ has a loop, while the original definition in \cite[Theorem 2.2]{EPW-2016} applies only to loopless matroids.

Gedeon, Proudfoot and Young \cite{GPY-2017} proved that for any equivariant matroid $W\curvearrowright M$  there is a unique way to define
the equivariant Kazhdan-Lusztig polynomial $P_{M}^W(t)\in \VRep(W)[t]$, such that the following conditions are satisfied:
\begin{itemize}
    \item If the ground set $E$ is empty, then $P^W_M(t)$ is the trivial representation in degree $0$;
    \item If $\rk M > 0$, then $\deg P^W_M(t) < \tfrac{1}{2}\rk M$;
    \item For every $M$,
    $$\displaystyle t^{\rk M} P^W_M(t^{-1}) = \sum_{[F]\in \mathcal{L}(M)/W}    \Ind_{W_F}^W\left(H_{M^F}^{W_F}(t) \otimes P^{W_F}_{M_F}(t)\right),$$
    where $W_F\subseteq W$ is the stabilizer of $F$ and $\mathcal{L}(M)/W$ denotes the set of orbits of the natural action of $W$ on $\mathcal{L}(M)$.
\end{itemize}
Here we adopt the definition of the equivariant Kazhdan-Lusztig polynomial presented in  \cite{GXY-2021}.
%is different from that in \cite[Theorem 2.8]{GPY-2017}.

%On the one hand, the fourth condition in the original definition
%is superfluous and omitted here. On the other hand, the original definition of
%$P_{M}^W(t)$ applies only to loopless matroids as in the ordinary case, see also \cite[Remark 2.2]{KNPV-2022}.

As pointed out by Gedeon, Proudfoot and Young in their paper \cite{GPY-2017},
the equivariant Kazhdan-Lusztig polynomials provide more powerful tools to study the
ordinary Kazhdan-Lusztig polynomials. In fact, if $W$ is a trivial group, then $P_M^W(t)$ is just $P_M(t)$. Moreover, we can also recover $P_M(t)$ by taking dimension to the coefficients of $P_M^W(t)$.
Since the dimension of a honest representation is a positive number, we can deduce the positivity of the coefficients of ordinary Kazhdan-Lusztig polynomials from that of the corresponding equivariant Kazhdan-Lusztig polynomials by sending a honset representation to its dimension.

We proceed to prove the induced log-concavity of $P_{U_{m,d}}^{\mathfrak{S}_{m+d}}(t)$ and $P^{\GL_{m+d}(\mathbb{F}_q)}_{U_{m,d}(q)}(t)$, which implies the log-concavity of $P_{U_{m,d}}(t)$ and $P_{U_{m,d}(q)}(t)$.
To this end, we need the following result recently established by Gao, Xie, and Yang \cite{GXY-2021}.

\begin{thm}[{\cite[Theorem 3.7]{GXY-2021}}]\label{main-equi-schur-skew}
    For any $d\geq 1$ and $m \geq 0$, we have
    \begin{align}\label{eq-skew-version}
        P_{U_{m,d}}^{\mathfrak{S}_{m+d}}(t)=\sum_{i=0}^{\lfloor (d-1)/2 \rfloor}
        V_{(m+d-2i,{(d-2i+1)}^i) / ( {(d-2i-1)}^i)} t^i.
    \end{align}
\end{thm}

%Motivated by Gedeon, Proudfoot and Young's positivity conjecture for equivariant Kazhdan-Lusztig polynomials \cite[Conjecture 2.13]{GPY-2017},
Proudfoot \cite{P-2019}
studied the equivariant Kazhdan-Lusztig polynomial $P^{\GL_{m+d}(\mathbb{F}_q)}_{U_{m,d}(q)}(t)$ and obtained that
\begin{align}
    P^{\GL_{m+d}(\mathbb{F}_q)}_{U_{m,d}(q)}(t)=
    \sum_{i=0}^{\lfloor (d-1)/2 \rfloor}\left(\sum_{b=1}^{\min\{m,d-2i\}} V_{\lambda}(q)\right)t^i, \label{eq-equivariant-kl-qniform}
\end{align}
%\begin{align}
%    P^{\GL_{m+d}(\mathbb{F}_q)}_{U_{m,d}(q)}(t)=
%    \sum_{i=0}^{\lfloor (d-1)/2 \rfloor}C_{m,d}^i(q)t^i, \qquad
%    C_{m,d}^i(q)=\sum_{b=1}^{\min\{m,d-2i\}} V_{\lambda}(q), \label{eq-equivariant-kl-qniform}
%\end{align}
where $\lambda=(m+d-2i-b+1,b+1,2^{i-1})$ is a partition of $m+d$ and
$V_\lambda(q)$ denotes the associated irreducible unipotent representation of $\GL_{m+d}(\mathbb{F}_q)$.
By the Comparison theorem and \eqref{eq-skew-version}, one can rewrite \eqref{eq-equivariant-kl-qniform} as
\begin{align}\label{eq-equivariant-kl-qniform-skew}
    P^{\GL_{m+d}(\mathbb{F}_q)}_{U_{m,d}(q)}(t)
=\sum_{i=0}^{\lfloor (d-1)/2 \rfloor}
        V_{(m+d-2i,{(d-2i+1)}^i) / ( {(d-2i-1)}^i)}(q) t^i,
\end{align}
where
$$
V_{(m+d-2i,{(d-2i+1)}^i) / ( {(d-2i-1)}^i)}(q)=\sum_{b=1}^{\min\{m,d-2i\}} V_{\lambda}(q).
$$
By taking dimensions of the representations in \eqref{eq-equivariant-kl-qniform-skew},
one can directly obtain that
\begin{align*}%\label{eq-pumdq}
    P_{U_{m,d}(q)}(t)=\sum_{i=0}^{\lfloor\frac{d}{2}\rfloor} c_{m,d}^i (q)t^i, \qquad
    c_{m,d}^i(q)=\dim V_{(m+d-2i,{(d-2i+1)}^i) / ( {(d-2i-1)}^i)}(q).
\end{align*}

The main result of this subsection is as follows.

\begin{thm}\label{thm-KL-qumd}
    For any prime power $q$, and any nonnegative integers $m\geq 0$ and $d\geq 1$, the equivariant Kazhdan-Lusztig polynomials $P_{U_{m,d}}^{\mathfrak{S}_{m+d}}(t)$ and $P^{\GL_{m+d}(\mathbb{F}_q)}_{U_{m,d}(q)}(t)$ are inductively log-concave.
\end{thm}

\begin{proof}
    According to Proposition \ref{prop-unirep-to-irrdrep}, it is sufficient to prove the induced log-concavity of $P_{U_{m,d}}^{\mathfrak{S}_{m+d}}(t)$.
    Since the image of $P_{U_{m,d}}^{\mathfrak{S}_{m+d}}(t)$ under the Frobenius map is
        \begin{align*}
            \ch\,P_{U_{m,d}}^{\mathfrak{S}_{m+d}}(t)
            =\sum_{i=0}^{\lfloor (d-1)/2 \rfloor}s_{(m+d-2i,(d-2i+1)^i) / ((d-2i-1)^i)}t^{i},
        \end{align*}
    by Proposition \ref{prop-il-to-sp} it suffices to show that for any $1\le i\le \lfloor (d-1)/2 \rfloor-1$ there holds
    $$([t^i](\ch\,P_{U_{m,d}}^{\mathfrak{S}_{m+d}}(t)))^2\geq_s
    ([t^{i+1}](\ch\,P_{U_{m,d}}^{\mathfrak{S}_{m+d}}(t)))\cdot ([t^{i-1}](\ch\,P_{U_{m,d}}^{\mathfrak{S}_{m+d}}(t))),$$
    namely,
    \begin{align*}
        s^2_{(m+d-2i,(d-2i+1)^i) / ((d-2i-1)^i)}
    &\ge_{s}s_{(m+d-2(i+1),(d-2(i+1)+1)^{(i+1)}) / ((d-2(i+1)-1)^{(i+1)})}\\
    &\qquad \times s_{(m+d-2{(i-1)},(d-2(i-1)+1)^{(i-1)}) / ((d-2{(i-1)}-1)^{(i-1)})},
    \end{align*}
    which follows from Theorem \ref{thm-Schur-main} by taking $n=m+d,\,k=-2,\,a=d+1,\,b=d-1,\,c=1,\,d=0$.
    So the proof is complete.
\end{proof}

By taking dimensions of representations in Theorem \ref{thm-KL-qumd}, we obtain the log-concavity of $P_{U_{m,d}(q)}(t)$ and recover the log-concavity of $P_{U_{m,d}}(t)$.
This provides further evidence for Elias, Proudfoot and Wakefield's log-concavity conjecture on matroid Kazhdan-Lusztig polynomials \cite[Conjecture 2.5]{EPW-2016}.

Motivated by Theorem \ref{thm-KL-qumd}, we have the following conjecture, which has been verified for $1\le m,d\le 15$ by using SageMath \cite{SAGE}.

\begin{conj}\label{conj-sil-klpol}
    For any prime power $q$, and any nonnegative integers $m\geq 0$ and $d\geq 1$, both $P_{U_{m,d}}^{\mathfrak{S}_{m+d}}(t)$ and $P^{\GL_{m+d}(\mathbb{F}_q)}_{U_{m,d}(q)}(t)$ are strongly inductively log-concave.
\end{conj}

\subsection{Equivariant inverse Kazhdan-Lusztig polynomials}\label{subsect-4.3}

The main objective of this subsection is to prove the induced log-concavity of
equivariant inverse Kazhdan-Lusztig polynomials of uniform matroids and \q-niform matroids.

Recall that the inverse Kazhdan-Lusztig polynomial of a matroid $M$, denoted by $Q_M(t)$, was introduced by Gao and Xie \cite{GX-2021} from the viewpoint of incidence algebras.
Proudfoot \cite{P-2021} noted that the equivariant inverse Kazhdan-Lusztig polynomial of an equivariant matroid $W\curvearrowright M$, denoted by $Q_M^W(t)$,
can be similarly defined by applying the equivariant Kazhdan-Lusztig-Stanley theory. Following Braden, Huh, Matherne, Proudfoot and Wang \cite{BHMPW-2020}, the polynomial $Q_M^W(t)$ can be defined inductively as follows:
\begin{itemize}
\item If the ground set of $M$ is empty, then  $Q_M^W(t)$ is the trivial representation in degree $0$;

\item Otherwise we have
    \begin{align}\label{eq-equivariant-inverse-klpol}
        Q_M^W(t)=-\sum_{\emptyset\neq F\in\mathcal{L}(M)}(-1)^{\rk F}\frac{|W_F|}{|W|}\Ind_{W_F}^{W}\left(P_{M^F}^{W_F}(t)\otimes Q_{M_F}^{W_F}(t)\right).
    \end{align}
\end{itemize}
Let $C_{M,W}^i$ be the coefficient of $t^i$ in $P_{M}^W(t)$ and $D_{M,W}^i$ be the coefficient of $t^i$ in $Q_{M}^W(t)$. Taking the coefficients of $t^i$ in \eqref{eq-equivariant-inverse-klpol} leads to
\begin{align}\label{coeff-inverse-and-kl}
    D_{M,W}^i=-\sum_{\emptyset\neq F\in\mathcal{L}(M)\atop 0\le j\le \rk F}(-1)^{\rk F}\frac{|W_F|}{|W|}\Ind_{W_F}^{W}\left(C_{M^F,W_F}^j\otimes D_{M_F,W_F}^{i-j}\right).
\end{align}

The equivariant inverse Kazhdan-Lusztig polynomials of uniform matroids were studied by
Gao, Xie, and Yang \cite{GXY-2021}, who obtained the following result.

\begin{thm}[{\cite[Theorem 3.2]{GXY-2021}}]\label{thm-inver-equivariant-umd}
    For any integers $m\ge 0$ and $d\ge 1$, we have
    \begin{equation}\label{def-eikl}
        Q_{U_{m,d}}^{\mathfrak{S}_{m+d}}(t)=\sum_{i=0}^{\lfloor (d-1)/2 \rfloor}V_{(m+1,2^i,1^{d-2i-1})}t^i.
    \end{equation}
 \end{thm}

Along the lines of the proof of \cite[Theorem 1.2]{P-2019}, we proceed to show that the equivariant inverse Kazhdan-Lusztig polynomials of equivariant $q$-niform matroids are given as follows.

\begin{thm}\label{formula-inverse-equivariant-kl-q}
    For any equivariant $q$-niform matroid $\GL_{m+d}(\mathbb{F}_q)\curvearrowright U_{m,d}(q)$ with $d\geq 1$ and $m\geq 0$, we have
    \begin{align} \label{def-eikl-q}
        Q_{U_{m,d}(q)}^{\GL_{m+d}(\mathbb{F}_q)}(t)=\sum_{i=0}^{\lfloor (d-1)/2 \rfloor}
                                        V_{(m+1,2^{i},1^{d-2i-1})}(q)t^i.
    \end{align}
\end{thm}

\begin{proof} Note that the right-hand side of \eqref{def-eikl-q} is just the unipotent $q$-analogue of that of \eqref{def-eikl} by the Comparison Theorem.
For notational convenience, let
\begin{align*}
D_{m,d}^i&=D_{U_{m,d},\mathfrak{S}_{m+d}}^i,\, C_{m,d}^i=C_{U_{m,d},\mathfrak{S}_{m+d}}^i,\\%\label{eq-dmi}\\
D_{m,d}^i(q)&=D_{U_{m,d}(q),\GL_{m+d}(\mathbb{F}_q)}^i,\, C_{m,d}^i(q)=C_{U_{m,d}(q),\GL_{m+d}(\mathbb{F}_q)}^i.%\label{eq-dmiq}
\end{align*}

We first derive a recursive formula for $D_{m,d}^i$ from \eqref{coeff-inverse-and-kl}.
To this end, let us analyse the flats of $U_{m,d}$. For the unique flat $F$ of $U_{m,d}$ with $\rk F=d$, the localization $M^F$ is just $U_{m,d}$ and the ground set of the contraction $M_F$ is empty. For any proper flat $F$ with $\rk F=k\le d-1$, we have ${(U_{m,d})}_F\cong U_{m,d-k}$
    and ${(U_{m,d})}^F\cong U_{0,k}$. Note that the stabilizer $W_F$ of a flat $F$ of cardinality $k$ is isomorphic to the Young subgroup
    $\mathfrak{S}_k\times \mathfrak{S}_{m+d-k}$. Thus, for uniform matroid $U_{m,d}$, \eqref{coeff-inverse-and-kl} transforms into the following recursion:
    \begin{eqnarray}\label{equ-Q-umd}
        D_{m,d}^i
        & =& -\sum_{k=1}^{d-1}\sum_{j=0}^k(-1)^{k}\Ind_{\mathfrak{S}_{k}\times \mathfrak{S}_{m+d-k}}^{\mathfrak{S}_{m+d}}
        \left(C_{0,k}^j\boxtimes D_{m,d-k}^{i-j}\right) + C_{m,d}^i\nonumber\\
        & =& -\sum_{k=1}^{d-1}\sum_{j=0}^k(-1)^{k}C_{0,k}^j * D_{m,d-k}^{i-j} + C_{m,d}^i.
    \end{eqnarray}

Next we derive a recursive formula for $D_{m,d}^i(q)$ from \eqref{coeff-inverse-and-kl}.
Note that proper flats of $U_{m,d}(q)$ are collections of linearly independent hyperplanes
    in $F_q^{m+d}$ of cardinality less than $d$.
    If $F$ is a proper flat of $U_{m,d}(q)$ with $\rk F=k$, then
    $U_{m,d}(q)_F\cong U_{m,d-k}(q)$ and $U_{m,d}(q)^F\cong U_{0,k}(q)$.
    The stabilizer of a proper flat $F$ of cardinality $k$ is isomorphic to the parabolic subgroup  $P_{k,m+d}(\mathbb{F}_q)\subseteq\GL_{m+d}(\mathbb{F}_q)$. Thus,  for $q$-niform matroid $U_{m,d}(q)$, \eqref{coeff-inverse-and-kl} transforms into the following recursion:
    \begin{align}\label{equ-Q-qmd}
        D_{m,d}^i(q)
        & = -\sum_{k=1}^{d-1}\sum_{j=0}^k(-1)^{k}\Ind_{P_{k,m+d}(\mathbb{F}_q)}^{\GL_{m+d}(\mathbb{F}_q)}
        \left(C_{0,k}^j(q){\boxtimes} D_{m,d-k}^{i-j}(q)\right)
        + C_{m,d}^i(q)\nonumber\\
        & = -\sum_{k=1}^{d-1}\sum_{j=0}^k(-1)^{k}C_{0,k}^j(q) * D_{U_{m,d-k}}^{i-j}(q)
        + C_{U_{m,d}}^i(q).
    \end{align}
    As shown by Proudfoot \cite{P-2019}, each coefficient $C_{U_{m,d}}^i(q)$ is the unipotent $q$-analogue of the corresponding $C_{U_{m,d}}^i$.
    By the Comparison Theorem, the recursive formula \eqref{equ-Q-umd} for $D_{m,d}^i$ is essentially the same as the recursive formula \eqref{equ-Q-qmd} for $D_{m,d}^i(q)$.
    This completes the proof.
\end{proof}

The main result of this subsection is as follows.

\begin{thm}\label{thm-inverse-umq-qumd}
    For any $d\ge 1,m\ge 0$ and prime power $q$, both $Q_{U_{m,d}(q)}^{\GL_{m+d}(\mathbb{F}_q)}(t)$ and $Q_{U_{m,d}}^{\mathfrak{S}_{m+d}}(t)$ are inductively log-concave.
    % let $D_{m,d}^i(q)$ be as defined in \eqref{eq-dmiq}. Then for any $0< i< \lfloor\frac{d-1}{2}\rfloor$ we have
    % \begin{align}\label{eq-inverse-dmiq}
    %     D_{m,d}^i(q)*D_{m,d}^i(q)-D_{m,d}^{i-1}(q)*D_{m,d}^{i+1}(q)\in \Rep(\GL_{m+d}(\mathbb{F}_q)).
    % \end{align}
\end{thm}

\begin{proof} The proof is very similar to that of Theorem \ref{thm-KL-qumd}.
    By Proposition \ref{prop-unirep-to-irrdrep}, it is sufficient to show the induced log-concavity of $Q_{U_{m,d}}^{\mathfrak{S}_{m+d}}(t)$.
    Since the image of $Q_{U_{m,d}}^{\mathfrak{S}_{m+d}}(t)$ under the Frobenius map is
        \begin{align*}
            \ch\,Q_{U_{m,d}}^{\mathfrak{S}_{m+d}}(t)
            =\sum_{i=0}^{\lfloor (d-1)/2 \rfloor}s_{(m+1,2^{i+1},1^{d-2i-3})}t^{i}.
        \end{align*}
    Thus, by Proposition \ref{prop-il-to-sp}, it suffices to show that
   \begin{align*}%\label{sequ-Schur-iumd}
    s^2_{(m+1,2^i,1^{d-2i-1})}\ge_ss_{(m+1,2^{i-1},1^{d-2i+1})} s_{(m+1,2^{i+1},1^{d-2i-3})}.
   \end{align*}
    One can verify that this is a special case of Theorem \ref{thm-LPP-sort} with $\lambda=(m+1,2^{i-1},1^{d-2i+1})$, $\nu=(m+1,2^{i+1},1^{d-2i-3})$ and $\mu=\rho=\emptyset$.
    The proof is complete.
\end{proof}

Gao and Xie \cite{GX-2021} conjectured that for any matroid $M$ the inverse Kazhdan-Lusztig polynomial $Q_M(t)$ is log-concave. The above theorem allows us to obtain
the log-concavity of $Q_{U_{m,d}}(t)$ and the log-concavity of $Q_{U_{m,d}(q)}(t)$ simultaneously, the former of which has been proved by Gao and Xie  \cite[Proposition 4.3]{GX-2021}, and the latter seems to be new. This provides some positive evidence for Gao and Xie's log-concavity conjecture on the inverse Kazhdan-Lusztig polynomials.

\begin{cor}
For a fixed prime power $q$, any $d\ge 1$ and $m\ge 0$, both $Q_{U_{m,d}(q)}(t)$ and $Q_{U_{m,d}}(t)$ are log-concave.
\end{cor}

Motivated by Theorem \ref{thm-inverse-umq-qumd}, we  propose the following conjecture,  similar to Conjecture \ref{conj-sil-klpol}.
This conjecture has been verified for $1\le m,d\le 15$.

\begin{conj}
    For $d\ge 1,m\ge 0$ and prime powers $q$, both $Q_{U_{m,d}(q)}^{\GL_{m+d}(\mathbb{F}_q)}(t)$ and $Q_{U_{m,d}}^{\mathfrak{S}_{m+d}}(t)$ are strongly inductively log-concave.
\end{conj}

\begin{rem}
It is well known that a log-concave sequence of nonnegative numbers without internal zeros is also strongly log-concave. However, an inductively log-concave sequence of representations without internal zeros need not {to} be strongly inductively log-concave.
Let us consider the polynomial
\begin{align*} J(t)=4V_{(2)}+\left(2V_{(2)}+4V_{(1,1)}\right)t+\left(V_{(2)}+4V_{(1,1)}\right)t^2+4V_{(1,1)}t^3\in \Rep(\mathfrak{S}_2)[t].
\end{align*}
The image of $J(t)$ under the Frobenius characteristic map is
\begin{align*}
\ch\,J(t)=4s_{(2)}+\left(2s_{(2)}+4s_{(1,1)}\right)t+\left(s_{(2)}+4s_{(1,1)}\right)t^2+4s_{(1,1)}t^3
\end{align*}
By using SageMath, we find that
\begin{align*}
    (2s_{(2)}+4s_{(1,1)})^2-4s_{(2)}\times (s_{(2)}+4s_{(1,1)})
    =16s_{(1, 1, 1, 1)} + 16s_{(2, 1, 1)} + 16s_{(2, 2)} \geq_s 0
\end{align*}
and
\begin{align*}
    \left(s_{(2)} +4s_{(1,1)}\right)^2-\left(2s_{(2)}+4s_{(1,1)}\right)\times 4s_{(1,1)}
    =s_{(2, 2)} + s_{(3, 1)} + s_{(4)} \geq_s 0.
\end{align*}
So the polynomial $J(t)$ is inductively log-concave with respect to the triple $(\mathfrak{S}_2\times \mathfrak{S}_2, \mathfrak{S}_4, \mathrm{id})$.
However, we find that the difference
\begin{align*}
    \left(2s_{(2)}+4s_{(1,1)}\right)&\times \left(s_{(2)}+4s_{(1,1)}\right)-4s_{(2)}\times 4s_{(1,1)}\\
    &=16s_{(1, 1, 1, 1)} + 12s_{(2, 1, 1)} + 18s_{(2, 2)} \mathbf{-} 2s_{(3, 1)} + 2s_{(4)}
\end{align*}
is not $s$-positive.
So the polynomial $J(t)$ is not strongly inductively log-concave with respect to $(\mathfrak{S}_2\times \mathfrak{S}_2, \mathfrak{S}_4, \mathrm{id})$.
\end{rem}

\subsection{Equivariant $Z$-polynomials} \label{subsect-4.4}

In this subsection we aim to study the induced log-concavity of
equivariant $Z$-polynomials of uniform matroids and \q-niform matroids.

The concept of $Z$-polynomial was introduced by Proudfoot, Xu and Young \cite{PXY-2018}.
Given a matroid $M$, its $Z$-polynomial, denoted by $Z_M(t)$, is defined by
$$Z_M(t):=\sum_{F\in\mathcal{L}(M)}t^{\rk F}P_{M_F}(t).$$
Proudfoot, Xu and Young showed that $Z$-polynomials are palindromic and found that this symmetry is very helpful for computing Kazhdan-Lusztig polynomials of matroids. They also introduced
the equivariant version of $Z$-polynomial for an equivariant matroid $W\curvearrowright M$, denoted by $Z_{M}^W(t)$, as
\begin{align}\label{ZUWM}
    Z_{M}^W(t)=\sum_{F\in\mathcal{L}(M)}\frac{|W_F|}{|W|}\Ind_{W_F}^{W} (P_{M_F}^{W_F}(t))  t^{\rk F}\in \VRep(W)[t].
\end{align}
Proudfoot, Xu and Young conjectured that for any matroid $M$ the $Z$-polynomial $Z_M(t)$ is log-concave, and for any equivariant matroid $W\curvearrowright M$
the equivariant $Z$-polynomial $Z_{M}^W(t)$ is equivariantly log-concave.

Proudfoot, Xu and Young \cite[Proposition 6.7]{PXY-2018} proved
the equivariant log-concavity of $Z_{U_{0,d}(q)}^{\GL_{d}(\mathbb{F}_q)}(t)$.
%{Wu, Xie and Zhang \cite{WXZ2023}} confirmed the log-concavity of $Z_{U_{m,d}}(t)$ in the same way as Xie and Zhang established the log-concavity of $P_{U_{m,d}}(t)$.
In this section we propose a unified approach to the log-concavity of $Z_{U_{m,d}}(t)$ and $Z_{U_{m,d}(q)}(t)$ from the viewpoint of induced log-concavity.
We first show that all coefficients of $Z_{U_{m,d}(q)}^{\GL_{m+d}(\mathbb{F}_q)}(t)$ only involve unipotent representations of
$\GL_{m+d}(\mathbb{F}_q)$. This is clear since by \eqref{eq-equivariant-kl-qniform}, \eqref{eq-equivariant-kl-qniform-skew} and \eqref{ZUWM} we  have
    \begin{align}
        &Z_{U_{m,d}(q)}^{\GL_{m+d}(\mathbb{F}_q)}(t)
        = \sum_{k=0}^{d-1}\Ind_{P_{m+d-k,k}(\mathbb{F}_q)}^{\GL_{m+d}(\mathbb{F}_q)} {(P_{U_{m,d-k}(q)}^{P_{m+d-k,k}(\mathbb{F}_q)}(t))} t^{k}+
        P_{\emptyset}^{\GL_{m+d}(\mathbb{F}_q)}(t)\cdot t^d\nonumber\\
        &= \sum_{k=0}^{d-1}\Ind_{P_{m+d-k,k}(\mathbb{F}_q)}^{\GL_{m+d}(\mathbb{F}_q)} (P_{U_{m,d-k}(q)}^{\GL_{m+d-k}(\mathbb{F}_q)}(t){\boxtimes} V_{(k)}(q)) t^{k}+
        V_{(m+d)}(q) t^d\nonumber\\
        & = \sum_{k=0}^{d-1}\sum_{i=0}^{\lfloor (d-k-1)/2 \rfloor}(
    V_{(m+d-k-2i,{(d-k-2i+1)}^i) / ( {(d-k-2i-1)}^i)}(q)*V_{(k)}(q)) t^{k+i}+
        V_{(m+d)}(q) t^d,\label{eq-equivariant-zpol}
    \end{align}
    where $V_{(k)}(q)$ is considered as the unipotent representation of $\GL_{k}(\mathbb{F}_q)$ indexed by the partition $(k)$.

We
have the following conjecture, which has been verified for $m,d\leq 15$ with the help of SageMath \cite{SAGE}.

\begin{conj}\label{conj-z-pol-il}
    For any $d\ge 1,m\ge 0$ and prime power $q$, both the polynomial  $Z_{U_{m,d}(q)}^{\GL_{m+d}(\mathbb{F}_q)}(t)$ and $Z_{U_{m,d}}^{\mathfrak{S}_{m+d}}(t)$ are strongly inductively log-concave.
\end{conj}

For the remainder of this section we will confirm the induced log-concavity conjecture for some special values of $m$ and $d$. Let $Z_{m,d}^i$ denote the coefficient of $t^i$ in $Z_{U_{m,d}}^{\mathfrak{S}_{m+d}}(t)$, and let $Z_{m,d}^i(q)$ denote the coefficient of $t^i$ in $Z_{U_{m,d}(q)}^{\GL_{m+d}(\mathbb{F}_q)}(t)$.
By definition, $Z_{U_{m,d}(q)}^{\GL_{m+d}(\mathbb{F}_q)}(t)$ is inductively log-concave if
\begin{align}\label{HClog-zmdq}
    Z_{m,d}^i(q)*Z_{m,d}^i(q)-Z_{m,d}^{i-1}(q)*Z_{m,d}^{i+1}(q)\in\Rep(\GL_{2m+2d}(\mathbb{F}^{2m+2d}_q)).
\end{align}
By Proposition \ref{prop-unirep-to-irrdrep}, the condition  \eqref{HClog-zmdq} is equivalent to
\begin{align*}%\label{HClog-zmd}
    Z_{m,d}^i*Z_{m,d}^i-Z_{m,d}^{i-1}*Z_{m,d}^{i+1}\in\Rep(\mathfrak{S}_{2m+2d}).
\end{align*}
Furthermore, by Proposition \ref{prop-il-to-sp}, this is equivalent to
    \begin{align}\label{spos-zmd}
        \ch\,Z_{m,d}^i\cdot \ch\,Z_{m,d}^i\geq_s \ch\,Z_{m,d}^{i-1}\cdot \ch\,Z_{m,d}^{i+1}.
    \end{align}
Thus, \eqref{spos-zmd} would imply the log-concavity of
$Z_{U_{m,d}}(t)$ and $Z_{U_{m,d}(q)}(t)$ simultaneously.

In order to prove \eqref{spos-zmd}, it is highly desirable to know the explicit value of each $\ch\,Z_{m,d}^i$. By the Comparison theorem and \eqref{eq-equivariant-zpol} we find that
    \begin{align*}
        Z_{U_{m,d}}^{\mathfrak{S}_{m+d}}(t)
        &= \sum_{k=0}^{d-1}\sum_{i=0}^{\lfloor (d-k-1)/2 \rfloor} (
    V_{(m+d-k-2i,{(d-k-2i+1)}^i) / ( {(d-k-2i-1)}^i)}*V_{(k)}) t^{k+i}+
        V_{(m+d)} t^d.
    \end{align*}
Further applying the Frobenius map leads to
    \begin{align*}
        \ch\, Z_{U_{m,d}}^{\mathfrak{S}_{m+d}}(t)=s_{(m+d)}t^d+\sum_{k=0}^{d-1}
        \sum_{j=0}^{\lfloor\frac{d-k-1}{2}\rfloor}
       s_{\left(m+d-k-2j,(d-k-2j+1)^j\right)/\left((d-k-2j-1)^j\right)} s_{(k)}t^{k+j}.
    \end{align*}
By the above formula, together with the symmetry of $Z_{U_{m,d}}^{\mathfrak{S}_{m+d}}(t)$, we get
\begin{align}\label{eq-coeff-s1}
    \ch\,Z_{m,d}^{d}&=\ch\,Z_{m,d}^0=s_{(m+d)},
\end{align}
and
for each $1\leq i\leq \lfloor\frac{d}{2}\rfloor$
\begin{align}\label{eq-coeff-s2}
\ch\,Z_{m,d}^i&=\ch\,Z_{m,d}^{d-i}= \sum_{k=d-2i+1}^{d-i}s_{(m+k-d+2i,(k-d+2i+1)^{d-i-k})/((k-d+2i-1)^{d-i-k})}s_{(k)}.
\end{align}

With \eqref{eq-coeff-s1} and \eqref{eq-coeff-s2} we are able to show the validity of \eqref{spos-zmd} for small $m$ and $d$, in support of Conjecture \ref{conj-z-pol-il}. Let us first consider the case of $m=0$.
Note that the uniform matroid $U_{0,d}$ is just the Boolean matroid $B_d$, and the $q$-niform matroid $U_{0,d}(q)$ is the matroid represented by all vectors in $\mathbb{F}_q^d$. We have the following result.

\begin{thm}\label{thm-Schur-pos-Z-Bd}
For any $d\geq 1$, the equivariant $Z$-polynomial $Z_{U_{0,d}(q)}^{\GL_{d}(\mathbb{F}_q)}(t)$ is strongly inductively log-concave, and so is $Z_{B_{d}}^{\mathfrak{S}_{d}}(t)$.
\end{thm}
\begin{proof}
From \eqref{eq-coeff-s1} and \eqref{eq-coeff-s2} we get that
    \begin{align*}
        \ch\, Z_{B_d}^{\mathfrak{S}_{d}}(t)=\sum_{k=0}^{d}s_{(d-k)}s_{(k)}t^k.
    \end{align*}
Now, for $0<i\leq j<d$, by \eqref{eq-jt-identity} we obtain
\begin{align*}
s_{(j,i)}&=s_{(i)}s_{(j)}-s_{(i-1)}s_{(j+1)},\\
s_{(d-i,d-j)}&=s_{(d-i)}s_{(d-j)}-s_{(d-i+1)}s_{(d-j-1)},
\end{align*}
from which there follows
\begin{align*}
    s_{(i)}s_{(d-i)}s_{(j)}s_{(d-j)}\geq_s s_{(i-1)}s_{(d-i+1)}s_{(j+1)}s_{(d-j-1)}.
\end{align*}
This completes the proof.
\end{proof}

Though it is difficult to prove \eqref{spos-zmd} for general $m,d,i$, we have the following partial result.

\begin{prop}\label{prop-first-3-terms}
    For any $m\ge 1$ and $d\ge 2$, we have
    \begin{align*}
        (\ch\,Z_{m,d}^{1})^2\ge_s \ch\,Z_{m,d}^{0}\cdot \ch\,Z_{m,d}^{2}\qquad \mbox{and} \qquad (\ch\,Z_{m,d}^{d-1})^2\ge_s \ch\,Z_{m,d}^{d}\cdot \ch\,Z_{m,d}^{d-2}.
    \end{align*}
\end{prop}

\begin{proof}
Due to the symmetry of $Z_{U_{m,d}}^{\mathfrak{S}_{m+d}}(t)$, we only need to prove
\begin{align}\label{eq-zmd-coeff-Schur}
(\ch\,Z_{m,d}^{d-1})^2\ge_s \ch\,Z_{m,d}^{d}\cdot \ch\,Z_{m,d}^{d-2}.
\end{align}

Let us first consider the case of $d=2$. In this case, we have
    \begin{align*}
        \ch\, Z_{U_{m,2}}^{\mathfrak{S}_{m+2}}(t)
        = & s_{(m+2)}t^2+s_{(m+1)}s_{(1)}t+s_{(m+2)}.
    \end{align*}
From Pieri's rule it follows that
\begin{align*}
s_{(m+1)}s_{(1)}=s_{(m+2)}+s_{(m+1,1)}.
\end{align*}
Now one can verify the validity of \eqref{eq-zmd-coeff-Schur} for $d=2$.

In the following we may assume that $d\geq 3$. By \eqref{eq-coeff-s1}, we have
    \begin{align*}
        \ch\,Z_{m,d}^{d} & =s_{(m+d)},\\
        \ch\,Z_{m,d}^{d-1} & =s_{(m+1)}s_{(d-1)},\\
        \ch\,Z_{m,d}^{d-2} & =s_{(m+2)}s_{(d-2)}+s_{(m+1,2)}s_{(d-3)}.
    \end{align*}
We proceed to show that
    \begin{align*}%\label{equ-schur-ge}
        s^2_{(m+1)}s^2_{(d-1)}&\ge_s s_{(m+d)}s_{(m+2)}s_{(d-2)}+s_{(m+d)}s_{(m+1,2)}s_{(d-3)}.
\end{align*}
for any $m\ge 1$ and $d\ge 3$.
By Pieri's rule, we have
\begin{align*}
    s_{(m+1,2)}=s_{(m+1)}s_{(2)}-s_{(m+2)}s_{(1)}.
\end{align*}
Thus it suffices to show that
    \begin{align*}%\label{equ-schur-ge}
        s^2_{(m+1)}s^2_{(d-1)}&\ge_s s_{(m+d)}s_{(m+2)}s_{(d-2)}+s_{(m+d)}s_{(d-3)}(s_{(m+1)}s_{(2)}-s_{(m+2)}s_{(1)})\\
    &= s_{(m+d)}s_{(d-3)}s_{(m+1)}s_{(2)}-s_{(m+d)}s_{(m+2)}(s_{(d-3)}s_{(1)}-s_{(d-2)})\\
    &= s_{(m+d)}s_{(d-3)}s_{(m+1)}s_{(2)}-s_{(m+d)}s_{(m+2)}s_{(d-3,1)},
\end{align*}
where the last equality follows from \eqref{eq-jt-identity}.
Now we only need to  show that
\begin{align*}
    s^2_{(m+1)}s^2_{(d-1)}\ge_s s_{(m+d)}s_{(d-3)}s_{(m+1)}s_{(2)}.
\end{align*}
By Theorem \ref{thm-LPP-main} we get
\begin{align*}
s^2_{(d-1)}\ge_s s_{(d-3)}s_{(d+1)}.
\end{align*}
Combining the above two inequalities,
it suffices to show that
\begin{align*}
s_{(m+1)}s_{(d+1)}\ge_s s_{(m+d)}s_{(2)}.
\end{align*}
This can be verified by applying Pieri's rule again in view of the condition $m\geq 1$ and $d\geq 3$.
The proof is complete.
\end{proof}

Finally, we prove the validity of Conjecture \ref{conj-z-pol-il} for small values of $d$.

\begin{thm}
For any $m\ge 1$ and $1\leq d\leq 5$, the equivariant $Z$-polynomial $Z_{U_{m,d}(q)}^{\GL_{m+d}(\mathbb{F}_q)}(t)$ is strongly inductively log-concave, and so is $Z_{U_{m,d}}^{\mathfrak{S}_{m+d}}(t)$.
\end{thm}

\begin{proof}
By virtue of the degree restriction of equivariant $Z$-polynomials and Proposition \ref{prop-first-3-terms}, we assume that $3\leq d\leq 5$ in the following.
The idea of our proof is to first write down an explicit formula of
$\ch\, Z_{U_{m,1}}^{\mathfrak{S}_{m+1}}(t)$ for each $3\leq d\leq 5$
by using \eqref{eq-coeff-s1} and \eqref{eq-coeff-s2}, and then to verify the equivalent Schur positivity \eqref{spos-zmd} for any $m\geq 1$.

%
%For $d=1$, we have
%\begin{align*}
%\ch\, Z_{U_{m,1}}^{\mathfrak{S}_{m+1}}(t)
%        = & s_{(m+1)}t+s_{(m+1)},
%\end{align*}
%and then there is nothing to prove.
%
%For $d=2$, we have
%\begin{align*}
%\ch\, Z_{U_{m,2}}^{\mathfrak{S}_{m+2}}(t)
%        = & s_{(m+2)}t^2+s_{(m+1)}s_{(1)}t+s_{(m+2)}.
%\end{align*}
%By Pieri's rule, we have
%    \begin{align*}
%        s_{(m+1)}s_{(1)}=s_{(m+2)}+s_{(m+1,1)},
%    \end{align*}
%and hence
%    \begin{align*}
%        (s_{(m+1)}s_{(1)})^2\ge_s (s_{(m+2)})^2,
%    \end{align*}
%as desired.

For $d=3$, we have
\begin{align*}
       \ch\, Z_{U_{m,3}}^{\mathfrak{S}_{m+3}}(t)
        = & s_{(m+3)}t^3 + s_{(m+1)} s_{(2)}t^{2} + s_{(m+1)} s_{(2)}t + s_{(m+3)}.
\end{align*}
By Proposition \ref{prop-first-3-terms}, it remains to prove
\begin{align*}
    (s_{(m+1)} s_{(2)})^2\geq_s (s_{(m+3)})^2.
\end{align*}
It is equivalent to prove
\begin{align*}
    s_{(m+1)} s_{(2)}\geq_s s_{(m+3)},
\end{align*}
which holds by Pieri's rule.
% By Pieri's rule, we have
%     \begin{align*}
%         s_{(m+1)}s_{(2)}=s_{(m+3)}+s_{(m+2,1)}+s_{(m+1,2)},
%     \end{align*}
% from which it follows that
%      \begin{align*}
%          (s_{(m+1)} s_{(2)})^2-s_{(m+3)}s_{(m+1)} s_{(2)}
%         %  &=(s_{(m+1)} s_{(2)}-s_{(m+3)})s_{(m+1)} s_{(2)}\\
%          &=(s_{(m+2,1)}+s_{(m+1,2)})s_{(m+1)} s_{(2)}
%      \end{align*}
%      and
%      \begin{align*}
%         (s_{(m+1)}s_{(2)})^-(s_{(m+3)})^2=(s_{(m+2,1)}+s_{(m+1,2)})^2+2s_{(m+3)}(s_{(m+2,1)}+s_{(m+1,2)})
%      \end{align*}
% are $s$-positive, as desired.

For $d=4$, we have
\begin{align*}
       \ch\, Z_{U_{m,4}}^{\mathfrak{S}_{m+4}}(t)
       = & s_{(m+4)}t^4+ s_{(m+1)} s_{(3)}t^3 + (s_{(m+2)} s_{(2)}+s_{(m+1,2)} s_{(1)})t^2+ s_{(m+1)} s_{(3)}t + s_{(m+4)}.
\end{align*}
We need to prove that
\begin{align*}
(s_{(m+1)} s_{(3)})^2 &\geq_s s_{(m+4)}(s_{(m+2)} s_{(2)}+s_{(m+1,2)} s_{(1)}),\\%\label{eq-um4-ss2-1}
(s_{(m+2)} s_{(2)}+s_{(m+1,2)} s_{(1)})^2 &\geq_s (s_{(m+1)} s_{(3)} )^2,\\%\label{eq-um4-ss2-2}
s_{(m+1)} s_{(3)}(s_{(m+2)} s_{(2)}+s_{(m+1,2)} s_{(1)}) &\geq s_{(m+4)}s_{(m+1)} s_{(3)},\\%\label{eq-um4-ss2-3}
(s_{(m+1)} s_{(3)})^2 &\geq (s_{(m+4)})^2.%\label{eq-um4-ss2-4}.
\end{align*}
Note that the first inequality %\eqref{eq-um4-ss2-1}
follows from Proposition \ref{prop-first-3-terms}.
To prove the rest, it suffices to show that
\begin{align}\label{eq-um4-ss3}
    s_{(m+2)} s_{(2)}+s_{(m+1,2)} s_{(1)}  \geq_s s_{(m+1)} s_{(3)} \geq_s s_{(m+4)}
\end{align}
% \begin{align}\label{eq-um4-ss3}
%         s_{(m+2)} s_{(2)}+s_{(m+1,2)} s_{(1)} & \geq_s s_{(m+1)} s_{(3)}\\
%         s_{(m+2)} s_{(2)}+s_{(m+1,2)} s_{(1)} &\geq_s s_{(m+4)}\\
%         s_{(m+1)} s_{(3)} &\geq s_{(m+4)}
% \end{align}
One can see that \eqref{eq-um4-ss3} holds for by  Pieri's rule.
% $m=1$. In the following, we may assume that $m\geq 2$. By Pieri's rule, we have
%     \begin{align*}
%         s_{(m+1)} s_{(3)}&=s_{(m+4)}+s_{(m+3,1)}+s_{(m+2,2)}+s_{(m+1,3)},\\
%         s_{(m+1,2)} s_{(1)}&=s_{(m+2,2)}+s_{(m+1,3)}+s_{(m+1,2,1)},\\
%         s_{(m+2)} s_{(2)}&=s_{(m+4)}+s_{(m+3,1)}+s_{(m+2,2)}.
%     \end{align*}
% Combining these identities leads to the validity of \eqref{eq-um4-ss3}.
This completes the proof of the case of $d=4$.

For $d=5$, we have
\begin{align*}
       \ch\, Z_{U_{m,5}}^{\mathfrak{S}_{m+5}}(t)
       = & s_{(m+5)}t^5+ s_{(m+1)}s_{(4)}t^4  + (s_{(m+2)}s_{(3)}+s_{(m+1,2)}s_{(2)})t^3\\
        +&( s_{(m+2)}s_{(3)}+s_{(m+1,2)}s_{(2)})t^2
         + s_{(m+1)} s_{(4)}t + s_{(m+5)}.
    %    \ch\, Z_{U_{m,6}}^{\mathfrak{S}_{m+6}}(t)
    %    = & s_{(m+6)}t^6+ (s_{(m+1)}s_{(5)})t^5 + (s_{(m+2)}s_{(4)}+ s_{(m+1,2)} s_{(3)})t^4 +(s_{(m+3)}s_{(3)}+s_{(m+2,3)/(1)} s_{(2)}\\
    %       &+s_{(m+1,2,2)} s_{(1)})t^3 +(s_{(m+2)}s_{(4)}+ s_{(m+1,2)} s_{(3)})t^2+(s_{(m+1)}s_{(5)})t+s_{(m+6)}\\
    %    \ch\, Z_{U_{m,7}}^{\mathfrak{S}_{m+7}}(t)
    %    = & s_{(m+7)}t^7+(s_{(m+1)}s_{(6)})t^6 + (s_{(m+2)}s_{(5)}+s_{(m+1,2)}s_{(4)})t^5 + (s_{(m+3)}s_{(4)}+s_{(m+2,3)/(1)}s_{(3)}\\
    %      &+s_{(m+1,2,2)}s_{(2)})t^4 +(s_{(m+3)}s_{(4)}+s_{(m+2,3)/(1)}s_{(3)}+s_{(m+1,2,2)}s_{(2)})t^3 \\
    %      &+(s_{(m+2)}s_{(5)}+s_{(m+1,2)}s_{(4)})t^2 +(s_{(m+1)}s_{(6)})t + s_{(m+7)}.
    \end{align*}
By Proposition \ref{prop-first-3-terms} it remains to show that
\begin{align*}%\label{eq-um5-ss2}
    (s_{(m+2)}s_{(3)}+s_{(m+1,2)}s_{(2)})^2 &\geq_s (s_{(m+1)}s_{(4)})^2,\\
    (s_{(m+2)}s_{(3)}+s_{(m+1,2)}s_{(2)})^2 &\geq_s s_{(m+1)}s_{(4)}( s_{(m+2)}s_{(3)}+s_{(m+1,2)}s_{(2)}),\\
    s_{(m+1)}s_{(4)}(s_{(m+2)}s_{(3)}+s_{(m+1,2)}s_{(2)}) &\geq_s s_{(m+5)}( s_{(m+2)}s_{(3)}+s_{(m+1,2)}s_{(2)}),\\
    s_{(m+1)}s_{(4)}( s_{(m+2)}s_{(3)}+s_{(m+1,2)}s_{(2)}) &\geq_s s_{(m+5)}s_{(m+1)} s_{(4)},\\
    (s_{(m+1)}s_{(4)})^2 &\geq_s (s_{(m+5)})^2.
\end{align*}
Now it suffices to prove
\begin{align*}%\label{eq-5-ss1}
    s_{(m+2)}s_{(3)}+s_{(m+1,2)}s_{(2)}\ge_s s_{(m+1)}s_{(4)}\ge_s s_{(m+5)},
\end{align*}
which can be verified by Pieri's rule.
% the above three identities.
This completes the proof of the case of $d=5$.
\end{proof}

%
%For example, for a polynomial $P_{U_{3,8}}^{\mathfrak{S}_{11}}(t)=a_0+a_1t+a_2t^2+a_3t^3$,
%where
%$$a_0=s_{(11)},\quad a_1=s_{(9,7)/(5)},\quad a_2=s_{(7,5,5)/(3,3)},\quad a_3=s_{(5,3,3,3)/(1,1,1)}.$$
%The induced log-concavity of $P_{U_{3,8}}^{\mathfrak{S}_{11}}(t)$ implies that
%\begin{align}\label{eq-an}
%    a_1^2\ge_s a_0a_2;\quad a_2^2\ge_s a_1a_3.
%\end{align}
%If we want to prove the strongly induced log-concavity of $P_{U_{3,8}}^{\mathfrak{S}_{11}}(t)$,
%we also need to show $a_1a_2\ge_sa_0a_3$, i.e.,
%\begin{align}\label{eq-ann}
%    s_{(9,7)/(5)}&s_{(7,5,5)/(3,3)} -  s_{(11)}s_{(5,3,3,3)/(1,1,1)}\ge_s 0.
%\end{align}
%It is not hard to compute \eqref{eq-ann}, however, we can not get the result the from \eqref{eq-an}.
%Therefore, the strongly induced log-concavity of a polynomial is stronger than the induced log-concavity of that polynomial.

\begin{rem}
%\textcolor{blue}
{Lots of matroids can be equipped with the action of symmetric groups}, such as paving matroids \cite{KNPV-2022}, braid matroids \cite{GPY-2017}, and thagomizer matroids \cite{XZ-2019}.
It is desirable to study the induced log-concavity of equivariant invariants of these matroids. However, we can not always obtain the induced log-concavity. For example, the equivariant Kazhdan-Lusztig polynomial $P_{B_7}^{\mathfrak{S}_7}(t)$ of the equivariant braid matroid of rank 6 is not inductively log-concave.
By using SageMath \cite{SAGE}, we find that
\begin{align*}
\mathrm{ch}\, P_{B_7}^{\mathfrak{S}_7}(t)=&s_{(7)}+(2s_{(7)}+2s_{(6,1)}+s_{(5,2)}+s_{(4,3)})t\\
 &+(2s_{(7)}+2s_{(6,1)}+2s_{(5,2)}+2s_{(4,3)}+2s_{(4,2,1)}+s_{(3,2,2)}+s_{(2,2,2,1)})t^2,
\end{align*}
while
%$$
%(2s_{(7)}+2s_{(6,1)}+s_{(5,2)}+s_{(4,3)})^2-s_{(7)}(2s_{(7)}+2s_{(6,1)}+2s_{(5,2)}+2s_{(4,3)}+2s_{(4,2,1)}+s_{(3,2,2)}+s_{(2,2,2,1)})$$
\begin{align*}
&(2s_{(7)}+2s_{(6,1)}+s_{(5,2)}+s_{(4,3)})^2\\
&-s_{(7)}(2s_{(7)}+2s_{(6,1)}+2s_{(5,2)}+2s_{(4,3)}+2s_{(4,2,1)}+s_{(3,2,2)}+s_{(2,2,2,1)})\\
&=4s_{(11, 1, 1, 1)}+3s_{(5, 5, 4)}+12s_{(7, 5, 1, 1)}+8s_{(7, 3, 3, 1)}+12s_{(6, 6, 2)}+15s_{(6, 5, 3)}+8s_{(6, 4, 4)}+6s_{(6, 6, 1, 1)}\\
&+10s_{(6, 5, 2, 1)}+11s_{(6, 4, 3, 1)}+4s_{(6, 4, 2, 2)}+3s_{(6, 3, 3, 2)}+4s_{(5, 5, 3, 1)}+3s_{(5, 4, 4, 1)}+2s_{(5, 5, 2, 2)}\\
&+4s_{(5, 4, 3, 2)}+s_{(5, 3, 3, 3)}+s_{(4, 4, 4, 2)}+s_{(4, 4, 3, 3)}-s_{(8, 2, 2, 1, 1)}-s_{(7, 2, 2, 2, 1)}+6s_{(10, 2, 1, 1)}+11s_{(9, 3, 1, 1)}\\
&+13s_{(8, 4, 1, 1)}
+13s_{(7, 4, 2, 1)}+s_{(9, 2, 2, 1)}+9s_{(8, 3, 2, 1)}-s_{(8, 2, 2, 2)}+2s_{(7, 3, 2, 2)}+10s_{(12, 1, 1)}+24s_{(7, 6, 1)}\\
&+18s_{(11, 2, 1)}
+8s_{(10, 2, 2)}+35s_{(8, 5, 1)}+27s_{(7, 5, 2)}+29s_{(10, 3, 1)}+36s_{(9, 4, 1)}+21s_{(9, 3, 2)}+30s_{(8, 4, 2)}\\
&+14s_{(8, 3, 3)}+22s_{(7, 4, 3)}+2s_{(14)}+8s_{(13, 1)}+14s_{(12, 2)}+20s_{(11, 3)}+25s_{(10, 4)}+25s_{(9, 5)}\\
&+20s_{(8, 6)}+8s_{(7, 7)},
\end{align*}
which is not $s$-positive.
\end{rem}
% \begin{align*}
%     a_1^2-a_0a_2 &=s_{(4, 4, 3, 3)} + s_{(4, 4, 4, 2)} + s_{(5, 3, 3, 3)} + 4s_{(5, 4, 3, 2)} + 3s_{(5, 4, 4, 1)} + 2s_{(5, 5, 2, 2)} + 4s_{(5, 5, 3, 1)} + 3s_{(5, 5, 4)} + 3s_{(6, 3, 3, 2)} + 4s_{(6, 4, 2, 2)} + 11s_{(6, 4, 3, 1)} + 8s_{(6, 4, 4)} + 10s_{(6, 5, 2, 1)} + 15s_{(6, 5, 3)} + 6s_{(6, 6, 1, 1)} + 12s_{(6, 6, 2)} - s_{(7, 2, 2, 2, 1)} + 2s_{(7, 3, 2, 2)} + 8s_{(7, 3, 3, 1)} + 13s_{(7, 4, 2, 1)} + 22s_{(7, 4, 3)} + 12s_{(7, 5, 1, 1)} + 27s_{(7, 5, 2)} + 24s_{(7, 6, 1)} + 8s_{(7, 7)} - s_{(8, 2, 2, 1, 1)}- s_{(8, 2, 2, 2)} + 9s_{(8, 3, 2, 1)} + 14s_{(8, 3, 3)} + 13s_{(8, 4, 1, 1)} + 30s_{(8, 4, 2)} + 35s_{(8, 5, 1)} + 20s_{(8, 6)} + s_{(9, 2, 2, 1)} + 11s_{(9, 3, 1, 1)} + 21s_{(9, 3, 2)} + 36s_{(9, 4, 1)} + 25s_{(9, 5)} + 6s_{(10, 2, 1, 1)} + 8s_{(10, 2, 2)} + 29s_{(10, 3, 1)} + 25s_{(10, 4)} + 4s_{(11, 1, 1, 1)} + 18s_{(11, 2, 1)} + 20s_{(11, 3)} + 10s_{(12, 1, 1)} + 14s_{(12, 2)} + 8s_{(13, 1)} + 2s_{(14)}
% \end{align*}

\noindent{\bf Acknowledgements.}
Alice Gao was supported by the National Science Foundation of China (11801447) and by the Natural Science Foundation of Shaanxi Province (2020JQ-104).
Matthew Xie was supported by the National Science Foundation of China (12271403).
Ethan Li was supported by the Fundamental Research Funds for the Central Universities (GK202207023).
Arthur Yang was supported in part by the Fundamental Research Funds for the Central Universities and the National Science Foundation of China (11522110 and 11971249).
We would like to thank Nicholas Proudfoot for his valuable correspondence and feedback.
We would also like to thank Tao Gui for the helpful discussions.

%\noindent{\bf Data Availability.} Data sharing is not applicable to this article as no data sets were generated or analysed during the current study.

%\section*{Declarations}
%\noindent{\bf Conflict of interest.}
%The authors declare that they have no conflict of interest to this work.


\begin{thebibliography}{1}

\bibitem{AHK-2018}
K. Adprasito, J. Huh, and E. Katz.
Hodge theory for combinatorial geometries,
{\em Ann. of Math.} {\bf 188}(2): 381--452, 2018.


%\bibitem{ALGV-2018}
%N. Anari, K. Liu, S. O. Gharan, and C. Vinzant.
%Log-concave polynomials (III): Mason's ultra-log-concavity conjecture for indepdent sets of matroids,
%{\em arXiv}:1811.01600, 2018.


% \bibitem{A-2018}
% S. Andrews.
% The unipotent modules of $\GL_n(q)$ via tableaux,
% {\em J. Algebraic Combin.} \textbf{47} no. 1: 1--15, 2018.

\bibitem{ADH-2020}
F. Ardila, G. Denham, and J. Huh.
Lagrangian geometry of matroids,
{\em J. Amer. Math. Soc.}, 2022, {https://doi.org/10.1090/jams/1009.}
% \textbf{36}1, 2022.

% \bibitem{BC-1972}
% C. T. Benson and C. W. Curtis.
% On the degrees and rationality of certain characters of finite Chevalley groups,
% {\em Trans. Amer. Math. Soc.} \textbf{165}: 251--273, 1972.

% \bibitem{BM-2004}
% F. Bergeron and P. McNamara.
% Some positive differences of products of Schur functions,
% {\em  arXiv:math/0412289v1}, 2004.

\bibitem{BBR-2005}
F. Bergeron, R. Biagioli, and M.H. Rosas.
Inequalities between Littlewood-Richardson coefficients,
{\em J. Combin. Theory Ser. A.} \textbf{113}: 567--590, 2004.

\bibitem{BTW-2006}
L.J. Billera, H. Thomas, and S. van Willigenburg.
Decomposable compositions, symmetric quasisymmetric functions and equality of ribbon Schur functions,
{\em Adv. Math.} \textbf{204}: 204--240, 2006.

% \bibitem{Bulter-1990}
% L. M. Butler.
% The $q$-log-concavity of $q$-binomial coefficients,
% {\em J. Comb. Theory Ser. A.} \textbf{54}: 54--63, 1990.

\bibitem{Branden-2014}
F. Br\"and\'en.
Unimodality, log-concavity, real-rootedness and beyond,
{\em Handbook of enumerative combinatorics.} CRC Press, 2015.

%\bibitem{BH-2018}
%P. Br\"and\'en and J. Huh.
%Hodge-Riemann relations for Potts model partition functions,
%{\em arXiv:1811.01696.} 2018.

\bibitem{BH-2020}
P. Br\"and\'en and J. Huh.
Lorenzian polynomials,
{\em Ann. of Math.} {\bf 192}(3): 821--891, 2020.

\bibitem{BHMPW-2020}
T. Braden, J. Huh, J. Matherne, N. Proudfoot, and B. Wang.
Singular Hodge theory for combinatorial geometries,
{\em arXiv:2010.06088}, 2020.

\bibitem{Brenti-1994}
F. Brenti.
Log-concave and unimodal sequences in algebra, combinatorics, and geometry.
{\em Contemp. Math.} \textbf{178}: 71--89, 1994.

% \bibitem{B-1999}
% F. Brenti.
% Twisted incidence algebras and Kazhdan-Lusztig-Stanley functions,
% {\em Adv. Math.} \textbf{148}(1): 44--74, 1999.

% \bibitem{B-2003}
% F. Brenti.
% P-kernels, IC bases and Kazhdan-Lusztig polynomials,
% {\em J. Algebra.} \textbf{259}(2): 613--627, 2003.

%\bibitem{Brylawski-1977}
%T. Brylawski.
%The broken-circuit complex,
%{\em Trans. Amer. Math. Soc.} \textbf{234}(2): 417--433, 1977.
%
%\bibitem{Brylawski-1982}
%T. Brylawski.
%The Tutte polynomial, Part 1: General theory,
%{\em Matroid theory and its applications} C.I.M.E. Summer Sch., \textbf{83}: 125--275, 1982.

%\bibitem{C-1987}
%C.J. Colbourn.
%The Combinatorics of Network Reliability,
%% {\em International Series of Monographs on Computer Science.} The Clarendon Press,
%Oxford University Press, New York, 1987.


\bibitem{C-1975}
C.W. Curtis.
Reduction theorems for characters of finite groups of Lie type,
{\em J. Math. Soc. Japan} \textbf{27}(4): 666--688, 1975.

% \bibitem{CWY-2011}
% W. Y. C. Chen, L. X. W. Wang and A. L. B. Yang.
% Recurrence relations for strongly q-logconvex polynomials,
% {\em Can. Math. Bull.} \textbf{54}(2): 217--229, 2011.

%\bibitem{D-1984}
%J. E. Dawson.
%A collection of sets related to the Tutte polynomial of a matroid,
%Graph theory, Proc. 1st Southeast Asian Colloq., Singapore 1983, {\em Lecture Notes in Mathematics}, \textbf{1073}: 193--204, 1984.

%\bibitem{DW-1974}
%T.A. Dowling and R.M. Wilson.
%The slimmest geometric lattices,
%{\em Trans. Amer. Math. Soc.} \textbf{196}: 203--215, 1974.
%
%
%\bibitem{DW-1975}
%T.A. Dowling and R.M. Wilson.
%Whitney number inequalities for geometric lattices,
%{\em Proc. Amer. Math. Soc.} \textbf{47}: 504--512, 1975.

\bibitem{EPW-2016}
B. Elias, N. Proudfoot, and M. Wakefield.
The Kazhdan-Lusztig polynomial of a matroid,
{\em Adv. Math.} \textbf{299}: 36--70, 2016.

\bibitem{EF-1999}
C.J. Eschenbrenner and M.J. Falk.
Orlik-Solomon algebras and Tutte polynomials,
\emph{J. Algebr. Comb.} \textbf{10}(2): 189--199, 1999.


% \bibitem{EW-2014}
% B. Elias, and G. Williamson.
% The Hodge theory of Soergel bimodules,
% {\em Ann. of Math.} \textbf{180}(3): 1089--1136, 2014.

%\bibitem{EH-2020}
%C. Eur and J. Huh.
%Logarithmic concavity for morphisms of matroids,
%{\em Adv. Math.} \textbf{367}(19), 2020.

% \bibitem{FNV-2021}
% L. Ferroni, G. D. Nasr and L. Vecchi.
% Stressed hyperplanes and Kazhdan-Lusztig gamma-positivity for matroids,
% {\em arXiv:2110.08869}, 2021.

\bibitem{FFLP-2005}
S. Fomin, W. Fulton, C.K. Li, and Y.T. Poon.
Eigenvalues, singular values, and Littlewood-Richardson coefficients,
{\em Amer. J. Math.} \textbf{127}: 101--127, 2005.


% \bibitem{HL-1983}
% R. B. Howlett, and G. I. Lehrer.
% Representations of generic algebras and finite groups of Lie type,
% {\em Trans. Amer. Math. Soc.} \textbf{280} no. 2: 753--779, 1983.

\bibitem{GLXYZ-2018}
A.L.L. Gao, L. Lu, M.H.Y. Xie, A.L.B. Yang, and P.B. Zhang.
The Kazhdan-Lusztig polynomials of uniform matroids,
{\em Adv. in Appl. Math.} \textbf{122}: 102117, 2021.

\bibitem{GX-2021}
A.L.L Gao and M.H.Y. Xie.
The inverse Kazhdan-Lusztig polynomial of a matroid,
{\em J. Combin. Theory Ser. B.} \textbf{151}: 375--392, 2021.

\bibitem{GXY-2021}
A.L.L. Gao, M.H.Y. Xie, and A.L.B. Yang.
The equivariant inverse Kazhdan-Lusztig polynomials of uniform matroids,
{\em SIAM J. Discret. Math.} \textbf{36}: 2553--2569, 2022.

% \bibitem{GXY-2019}
% A. L. L. Gao, M. H. Y. Xie, and A. L. B. Yang.
% Schur positivity and log-concavity related to longest increasing subsequences,
% {\em Discrete Math.}  \textbf{342}: 2570--2578, 2019.

\bibitem{GPY-2017}
K. Gedeon, N. Proudfoot, and B. Young.
The equivariant Kazhdan-Lusztig polynomial of a matroid,
{\em J. Combin. Theory Ser. A.} \textbf{150}: 267--294, 2017.

\bibitem{GPY-2017real}
K. Gedeon, N. Proudfoot, and B. Young.
Kazhdan-Lusztig polynomials of matroids: a survey of results and conjectures,
{\em S$\acute{e}$m. Lothar. Combin.} 78B, Article 80, 12 pp, 2017.

\bibitem{HRS-2021}
T. Hameister, S. Rao, and C. Simpson.
Chow rings of vector space matroids,
{\em J. Comb.} \textbf{12}: 55--83, 2021.

%\bibitem{Heron-1972}
%A. P. Heron.
%Matroid polynomials,
%{\em J. Combin. Theory Ser. B.} \textbf{16}: 248--254, 1974.
%
%\bibitem{Hoggar-1974}
%S. G. Hoggar.
%Chromatic polynomials and logarithmic concavility,
%{\em Math. Inst.} 164--202, 1972.


\bibitem{Huh-2012}
J. Huh.
Milnor numbers of projective hypersurfaces and the chromatic polynomial of graphs,
{\em J. Amer. Math. Soc.} \textbf{25}(3): 907--927, 2012.

%\bibitem{Huh-2015}
%J. Huh. $h$-vectors of matroids and logarithmic concavity,
%{\em Adv. Math.} \textbf{270}: 49--59, 2015.
%
\bibitem{HK-2012}
J. Huh and E. Katz.
Log-concavity of characteristic polynomials and the Bergman fan of matroids,
{\em Math. Ann.} \textbf{354}(3): 1103--1116, 2012.

% \bibitem{HSW-2021}
% J. Huh, B. Schr\"oter, and B. Wang.
% Correlation bounds for fields and matroids,
% {\em J. Eur. Math. Soc.} {\bf 24}: 1335--1351, 2022.

% \bibitem{K-2022}
% J. Q. Kathy.
% The $q$-log-concavity and unimodality of $q$-Kaplansky numbers,
% {\em Discrete Math.} \textbf{345} no 6. 2022.


\bibitem{KNPV-2022}
T. Karn, G. Nasr, N. Proudfoot, and L. Vecchi.
Equivariant Kazhdan-Lusztig theory of paving matroids,
arXiv:2202.06938.


\bibitem{KL-1979}
D. Kazhdan and G. Lusztig.
Representations of Coxeter groups and Hecke algebras,
{\em Invent. Math.} \textbf{53}: 165--184, 1979.

% \bibitem{Krattenthaler-1989}
% C. Krattenthaler.
% On the q-log-concavity of Gaussian binomial coefficients,
% {\em Monatsh. Math.} \textbf{107}: 333--339, 1989.

%\bibitem{KWWS-2008}
%R. C. King, T. A. Welsh, and S. J. van Willigenburg.
%Schur positivity of skew Schur function differences and applications to ribbons and Schubert classes,
%{\em J. Algebra Combin.} \textbf{28}(1): 139--167, 2008.

\bibitem{LPP-2007}
T. Lam, A. Postnikov, and P. Pylyavskyy.
Schur positive and Schur log-concavity,
{\em Amer. J. Math.} \textbf{129}: 1611--1622, 2007.

\bibitem{LLT-1997}
A. Lascoux, B. Leclerc, and J.Y. Thibon.
Ribbon tableaux, Hall-Littlewood symmetric functions, quantum affine algebras, and unipotent varieties,
{\em J. Math. Phys.} \textbf{38}(3): 1041--1068, 1997.

% \bibitem{LNR1-2019}
% K. Lee, G. D. Nasr, and J. Radcliffe.
% A combinatorial formula for Kazhdan-Lusztig polynomials of $\rho$-Removed Uniform matroids,
% {\em Electron. J. Combin.} \textbf{27} no. 4, Paper 4.7, 23 pp, 2019.

% \bibitem{LNR2-2020}
% K. Lee, G. D. Nasr, and J. Radcliffe.
% A combinatorial formula for Kazhdan-Lusztig polynomials of sparse paving matroids,
% {\em arXiv:2006.10209}, 2020.

\bibitem{Lenz-2013}
M. Lenz.
The $f$-vector of a representable-matroid complex is log-concave,
{\em Adv. in Appl. Math.} \textbf{51}(5): 543--545, 2013.

% \bibitem{Leroux-1990}
% P. Leroux.
% Reduced matrices and q-log-concavity properties of q-Stirling numbers,
% {\em J. Comb. Theory  Ser. A.} \textbf{54}: 64--84, 1990.

%\bibitem{Llamas-2010}
%A. Llamas and J. Mart\'inez-Bernal.
%Nested log-concavity,
%{\em Commun. Algebra.} \textbf{38}: 1968--1981, 2010.

\bibitem{LXY-2018}
L. Lu, M.H.Y. Xie, and A.L.B. Yang.
Kazhdan-Lusztig polynomials of fan matroids, wheel matroids and whirl matroids,
{\em J. Combin. Theory Ser. A.} \textbf{192}: 105665, 2022.

%\bibitem{Macdonald}
%I.G. Macdonald.
%Symmetric Functions and Hall Polynomials, 2nd edn. Oxford University Press, New York, 1995.
%
%2nd ed, with contribution
%by A.V. Zelevinsky and a foreword by Richard Stanley, reprint of the 2008 paperback
%edition [MR1354144]. Oxford Classic Texts in the Physical Sciences. The Clarendon
%Press, Oxford University Press, New York, 2015.

%\bibitem{Mason-1972}
%J.H. Mason.
%Matroids: unimodal conjectures and Motzkin's theorem, Combinatorics (Proc. Conf. Combinatorial
%Math., Math. Inst., Oxford, 1972), 207–220, {\em Inst. Math. Appl.}, Southend-on-Sea, 1972.

\bibitem{MMPR-2021}
J.P. Matherne, D. Miyata, N. Proudfoot, and E. Ramos.
Equivariant log concavity and representation stability,
{\em Int. Math. Res. Not.},  \textbf{2023}(5): 3885--3906, 2023.

\bibitem{O-1997}
A. Okounkov.
Log-concavity of multiplicities with applications to characters of $U(\infty)$,
{\em Adv. Math.} \textbf{127}(2): 258--282, 1997.

\bibitem{Oxley-2011}
J. Oxley.
Matroid Theory,
{\em Oxford Graduate Texts in Mathematics.} \textbf{21} Oxford University Press, Oxford, 2011.

\bibitem{Polo-1999}
P. Polo.
Construction of arbitrary Kazhdan-Lusztig polynomials in symmetric groups,
{\em Represent. Theory.} \textbf{3}: 90--104, 1999.

% \bibitem{P-2018}
% N. Proudfoot.
% The algebraic geometry of Kazhdan-Lusztig-Stanley polynomials,
% {\em EMS Surv. Math. Sci.}, \textbf{5}(1): 99--127, 2018.

\bibitem{P-2019}
N. Proudfoot.
Equivariant Kazhdan-Lusztig polynomials of \q-niform matroids,
{\em J. Algebra Combin.} \textbf{2}: 613--619, 2019.

\bibitem{P-2021}
N. Proudfoot.
Equivariant incidence algebras and equivariant Kazhdan-Lusztig-Stanley theory,
{\em Algebr. Comb.}, \textbf{4}(4): 675--681, 2021.

\bibitem{PXY-2018}
N. Proudfoot, Y. Xu, and B. Young.
The $Z$-polynomial of a matroid,
{\em Electron. J. Combin.} \textbf{25}(1), 2018.

%\bibitem{Read-1968}
%R. C. Read.
%An introduction to chromatic polynomials,
%{\em J. Comb. Theory} \textbf{4}(1): 52--71, 1968.

% \bibitem{R-1989}
% J. B. Remmel.
% A formula for the Kronecker products of Schur functions of hook shapes,
% {\em J. Algebra.} \textbf{120}(1): 100, 1989.

\bibitem{RSW-2007}
V. Reiner, K.M. Shaw, and S. Willigenburg.
Coincidences among skew Schur functions,
{\em Adv. Math.} \textbf{216}: 118--152, 2007.

% \bibitem{RS-2006}
% B. Rhoader, and M. Skandera.
% Kazhdan-Lusztig immanants and products of matrix minors,
% {\em J. Algebra.} \textbf{304}: 793--811, 2006.


%\bibitem{Rota-1964}
%G.-C. Rota.
%On the foundations of combinatorial theory. I. Theory of M\"obius functions,
%{\em Z. Wahrscheinlichkeitstheor. Verw. Geb.} \textbf{2}: 340--368, 1964.
%
%\bibitem{Rota-1970}
%G.-C. Rota.
%Combinatorial theory, old and new,
%{\em Actes Congr. internat. Math. 1970} \textbf{3}: 229--233, 1971.

% \bibitem{Sagan-1992-2}
% B. E. Sagan.
% Inductive proofs of q-log concavity,
% {\em Discrete Math.} \textbf{99}: 298-306, 1992.

% \bibitem{Sagan-1992}
% B. E. Sagan.
% Log-concave sequences of symmetric functions and analogs of the Jacobi-Trudi determinants,
% {\em Trans. Amer. Math. Soc.} \textbf{329} page: 795, 1992.

\bibitem{SAGE}
W.A. Stein, et al.
Sage Mathematics Software (Version 9.5), The Sage Development Team, http://www.sagemath.org, 2022.

\bibitem{Stanley-1989}
R.P. Stanley.
Log-concave and unimodal sequences in algebra, combinatorics, and geometry,
{\em Ann. New York Acad. Sci.} \textbf{576}: 500--534, 1989.


% \bibitem{Stanley-1992}
% R. P. Stanley.
% Subdivisions and local h-vectors.
% {\em J. Amer. Math. Soc.} 805--851, 1992.

\bibitem{Stanley}
R.P. Stanley.
Enumerative Combinatorics. Vol. 2, with a foreword by Gian-Carlo Rota and appendix 1 by Sergey Fomin,
Cambridge Studies in Advanced Mathematics, vol.62, Cambridge University Press, Cambridge, 1999.

% \bibitem{W-2018}
% M. Wakefield.
% A flag Whitney number formula for matroid Kazhdan-Lusztig polynomials,
% {\em Electron. J. Comb.} \textbf{25}(1): 1--22, 2018.

%\bibitem{Welsh-1976}
%D.J.A. Welsh.
%{\em Matroid Theory},
%Academic Press, London, New York, I. M. S. Monographs, No.8, 1976.

%\bibitem{WXZ2023}
%S.Y. Wu, M.H.Y. Xie, and P.B. Zhang. The ultra log-concavity of $Z$-polynomials and $\gamma$-polynomials of uniform matroids, preprint.

\bibitem{XZ-2019}
M.H.Y. Xie and P.B. Zhang.
Equivariant Kazhdan-Lusztig polynomials of thagomizer matroids,
{\em Proc. Amer. Math. Soc.} \textbf{147}: 4687--4695, 2019.

\bibitem{XZ-2021}
M.H.Y. Xie and P.B. Zhang.
The log-concavity of Kazhdan-Lusztig polynomials of uniform matroids,
{\em J. Syst. Sci. Complex.}  \textbf{36}: 117--128, 2023.
\end{thebibliography}
\end{document}